\documentclass[11pt]{article}
\usepackage{amsmath}
\usepackage{amssymb}
\usepackage{amsthm}
\usepackage{graphicx}
\usepackage{microtype}
\usepackage{hyperref}
\usepackage{titlesec}
\usepackage[titletoc,toc,title]{appendix}

\newtheorem{theorem}{Theorem}[section]
\newtheorem{lemma}[theorem]{Lemma}

\newtheorem{proposition}[theorem]{Proposition}
\newtheorem{corollary}[theorem]{Corollary}
\theoremstyle{definition}
\newtheorem{question}[theorem]{Question}

\newtheorem{definition}[theorem]{Definition}
\newtheorem{example}[theorem]{Example}

\newtheorem{remark}[theorem]{Remark}
\newtheorem*{Acknowledgements}{Acknowledgements}

\def\C{\mathcal{C}}
\def\d{\displaystyle}

\def\c{\gamma}
\def\x{\xi}

\def\L{\mathcal{L}}

\def\U{\mathcal{ULFP}}

\def\EE{TP}

\def\C(Z){C(Z^{\circ})}
\def\projection{TP }
\def\Projection{TP }
\def\PP{tp}
\def\Proj{\mathcal{TP}}
\def\Exp{\mathcal{I}}
\def\4{6}

\begin{document}
\title{\textbf{Weak tight geodesics in the curve complex: examples and gaps}}
\author{Yohsuke Watanabe}
\maketitle
\date
\begin{abstract}
Tight geodesics were introduced by Masur--Minsky in \cite{MM2}. They and their hierarchies have been a powerful tool in the study of the curve complex, mapping class groups, Teichm\"{u}ller spaces, and hyperbolic $3$--manifolds. In the same paper, they showed that there are at least one and at most finitely many tight geodesics between any two vertices in the curve complex. Bowditch found a uniform finiteness property on tight geodesics \cite{BO2}, and this property has given further important applications in some of the above studies. In this paper, we introduce classes of geodesics which are not tight but still have the uniform finiteness property. These classes of geodesics will be obtained as examples of weak tight geodesics which were introduced and shown to have the property in \cite{W2}. 
The aim of this paper is to study about weak tight geodesics focusing on giving examples of them with canonical constructions and investigating gaps between two classes of them. Our main investigation will be on weak tight geodesics contained in the class of $M$--weakly tight geodesics where $M$ is the bound given by the bounded geodesic image theorem. As $M$--weakly tight geodesics contain tight geodesics, the classes of weak tight geodesics to be introduced in this paper will live around tight geodesics. In appendices, we expand some of these studies to outside of the class of $M$--weakly tight geodesics.

\end{abstract}

\section{Introduction}
Let $S$ be a compact surface and $\xi(S)=3g+b-3$ where $g$ is the number of genus and $b$ is the number of the boundary components of $S$. Our main space in this paper is the curve complex $C(S)$ defined by Harvey \cite{HAR}. The vertices of $C(S)$ are free isotopy classes of simple, closed, essential, and non--peripheral curves in $S$. The simplices are spanned by collections of curves that are mutually disjoint. If $\x(S)=1$ then the simplices are spanned by the collections of curves that mutually intersect once if $S=S_{1,1}$ and twice if $S=S_{0,4}$. Our main focus will be on the $1$--skeleton of $C(S)$, called the curve graph. It is locally infinite and path--connected. Moreover, with the graph metric $d_{S}$, it is an infinite--diameter space \cite{MM1}. In general, there are infinitely many geodesics between a pair of vertices in $C(S)$.

In \cite{MM2}, Masur--Minsky introduced a class of geodesics called \emph{tight geodesics}. Tight geodesics and hierarchies of tight geodesics have been a powerful tool in the study of the curve complex, mapping class groups, Teichm\"{u}ller spaces, and hyperbolic $3$--manifolds.
Although, there are infinitely many geodesics between a pair of vertices in $C(S)$ in general, Masur--Minsky showed that there are at least one and only finitely many tight geodesics between any pair of vertices in $C(S)$ \cite{MM2}. In fact, due to Bowditch's work in \cite{BO2}, we know that tight geodesics satisfy a stronger property, see Theorem \ref{webb's} for more specific and general statements. We refer this property as Bowditch's uniform finiteness property for the rest of this paper. The property has various important applications; for instance, Bowditch showed that  mapping class groups action on the curve complex is acylindrical and that the stable length of pseudo--Anosov mapping classes are uniformly rational \cite{BO2}. Kida showed that the curve complex has Yu's property A \cite{Kida}. Bell--Fujiwara showed that the asymptotic dimension of the curve complex is finite \cite{BF}. This was a key result in a work of Bestvina--Bromberg--Fujiwara where they prove that the asymptotic dimension of mapping class groups is finite \cite{BBF}.

In \cite{W2}, the author introduced \emph{weak tight geodesics} and showed that they also satisfy Bowditch's uniform finiteness property by proving that the curve graphs behave like uniformly locally finite graphs under Masur--Minsky's subsurface projections, see $\S \ref{mol}$. However, only known example of weak tight geodesics so far is tight geodesics; in this paper, we find other examples of weak tight geodesics with canonical constructions, answering some questions asked in \cite{W2}. The existence of weak tight geodesics, which we present in this paper, easily follows from their constructions, and we distinguish them from tight geodesics by providing appropriate examples of pairs of vertices. The main class of weak tight geodesics will be \emph{\Projection geodesics}. They will arise from our investigation on the class of $M$--weakly tight geodesics where $M$ is the bound given by the Bounded Geodesic Image Theorem. Also, we study gaps between two classes of weak tight geodesics.

Before we recall notion of weak tight geodesics, we state some motivations for studying about weak tight geodesics from various aspects. 

In their recent work, Birman--Margalit--Menasco introduced \emph{efficient geodesics}: they showed that there are at least one and only finitely many efficient geodesics between any pair of vertices in $C(S)$ and that efficient geodesics are different from tight geodesics \cite{BMM}. However, it is not clear whether Bowditch's uniform finiteness property holds on efficient geodesics or not; hence it would be interesting to study whether efficient geodesics are an example of weak tight geodesics or not.

The definition of weak tight geodesics do not require surface topology. They are purely defined in terms of subsurface projection distances. Bestvina--Feighn introduced the notion of projections, called subfactor projections, on complexes which $Out(F_{n})$ acts on \cite{BEF2}. Therefore, the definition of weak tight geodesics makes sense there. Bell--Fujiwara only need Bowditch's uniform finiteness property for their finite asymptotic dimension result of the curve complex besides of the hyperbolicity; clearly one could use weak tight geodesics instead of tight geodesics in their proof. In fact, many $Out(F_{n})$--complexes are known to be hyperbolic, for instance, the free splitting complex due to Handel--Mosher \cite{HM} and the free factor complex due to Bestvina--Feighn \cite{BEF}. Hence, it would be interesting to apply the notion of weak tight geodesics to study their asymptotic dimensions, but this requires further studies on subfactor projections as they do not seem to have a complete list of projections such as annular projections. Some of these issues will be discussed in a forthcoming paper of the author. However, our construction of weak tight geodesics in this paper, which uses projections, could be useful there. 

Tight geodesics are shown to capture internal geometry of certain hyperbolic $3$--manifolds, which is due to Bowditch \cite{BO3} and Minsky \cite{M1}. Let $N$ denote a hyperbolic $3$--manifold which admits homotopy equivalence to $S$ with the correspondence between boundary components of $S$ and parabolic cusps of $N$. A curve in $S$ can be realized uniquely as a closed geodesic in $N$. Let $x,y\in C(S)$ and let $\{x=v_{0}, v_{1}, \cdots ,v_{d_{S}(x,y)}=y\}$ be a tight multigeodesic between them. Their length bounds theorem says that if the lengths of $x$ and $y$ are less than $k$ in $N$, then there exists $K(\x(S),k)$ such that the length of $v_{i}$ is less than $K$ in $N$ for all $i$. (We note that Minky's results are expressed in terms of hierarchies. Some relations between their results are stated in \cite[\S1]{BO3}.) Perhaps, it would be interesting to consider this result in the setting of weak tight geodesics where $K$ shall also depend on the classes of weak tight geodesics. However, for instance in Bowditch's proof, the proof requires a special property of ``tightness'' which is that if some curve intersects $v_{i}$ then it intersects $v_{i-1}$ or $v_{i+1}$. Weak tight geodesics do not have this property in general ($D$--weakly tight geodesics where $D>M$ never have it.). Nevertheless, the property is crucial in his proof, which we briefly indicate here: the proof adapts his proof of an earlier (than \cite{BO2}) finiteness result on tight geodesics from the same paper \cite{BO3}; it says that the set of curves contained in tight multigeodesics between $x$ and $y$ at time $i$-point is finite for all $i$. The proof is done by a geometric limit argument: fix a hyperbolic metric on $S$, and for the rest of this discussion, curves are assumed to be at their geodesic representatives in this metric. Suppose the statement is false, then there exists a sequence of multicurves $\{v_{p}^{n}\}$ which arise at time $p$--point on tight multigeodesics between $x$ and $y$ whose total length goes to infinity. Let $F_{p}$ denote the subsurface filled by the minimal geodesic laminations obtained as the limit of $\{v_{p}^{n}\}$ by Hausdorff convergence after passing to a subsequence if necessary. He shows that there exists a ``taut'' sequence of subsurfaces containing $F_{p}$ which can be used for a short cut on a tight multigeodesic between $x$ and $y$ by his ``2/3 lemma'', a contradiction. However, we emphasize that 2/3 lemma works only on a taut sequence, and the fact that he obtains a taut sequence via laminations from multicurves is because those multicurves are from tight multigeodesics which have the special property stated above. 
The proof of the length bounds result is also in this spirit, namely by a geometric limit argument supposing that the statement is false. However, it is a lot more complicated, so we discuss only to confirm that the special property is again crucial. For simplicity of the discussion, let us assume there are no accidental cusps in $N$. First, he observes ``tube penetration lemma" which says that there exists a Margulis constant $\eta(k, \x(S))$ so that if $T$ is a Margulis tube with the constant, then either $v_{i} \cap T=\emptyset$ or $v_{i}$ is the core curve of $T$ for all $i$; in this case, $v_{i}$ is said to be non--penetrating. Another important observation of his says that a sequence of non--penetrating multicurves give arise to a sequence of subsurfaces that can be used to find a short cut on a tight multigeodesic between $x$ and $y$, a contradiction. However, the proof of tube penetration lemma requires the special property. Hence the proof does not seem to apply in the setting of weak tight geodesics, at least directly.

\begin{Acknowledgements}
This work was not possible without \cite{W2}, which was written while the author was working under Ken Bromberg. The author is grateful to Ken Bromberg, without whose encouragement and support \cite{W2} would never have been possible.
\end{Acknowledgements}

\section{Tight geodesics and weak tight geodesics}
Before we recall the notion of weak tight geodesics from \cite{W2}, we recall tight geodesics theory from Masur--Minsky \cite{MM2}, Bowditch \cite{BO2}, and Webb \cite{WEB2}. 

\subsection{Tight geodesics}\label{mercy}
Tight geodesics were introduced by Masur--Minsky.
\begin{definition}[\cite{MM2}]\label{d}
Suppose $\x(S)=1.$ Every geodesic is defined to be a tight geodesic.
Suppose $\x(S)>1.$ Let $A\subseteq S$. We let $R(A)$ denote a regular neighborhood of $A$ and $F(A)$ denote the subsurface constructed by taking $R(A)$ and filling in every complementary component of $R(A)$ in $S$ which is a disk or a peripheral annulus. 
\begin{itemize}
\item A multicurve is a collection of curves which form a simplex in $C(S)$. A multigeodesic is a sequence of multicurves $\{v_{i}\}$ such that $d_{S}(a,b)=|p-q|$ for all $a\in v_{p},b\in v_{q}$, and for all $p,q$. A tight multigeodesic is a multigeodesic $\{v_{i}\}$ such that $v_{i}=\partial (F(v_{i-1} \cup v_{i+1}))$ for all $i$. 

\item Let $x,y \in C(S)$. A tight geodesic between $x$ and $y$ is a geodesic $\{u_{i}\}$ such that $u_{i}\in v_{i}$ for all $i$, where $\{v_{i}\}$ is a tight multigeodesic between $x$ and $y$.
\end{itemize}
\end{definition}

Masur--Minsky showed that a tight geodesic exists between any two vertices in $C(S)$, which follows by checking the existence of tight multigeodesics between two multicurves which are distance $3$ apart;

\begin{lemma}[\cite{MM2}]\label{construction}
Let $\{ v_{0},v_{1},v_{2},v_{3}\}$ be a multigeodesic. Let $v_{1}^{t}=\partial(F(v_{0}\cup v_{2}))$, furthermore let $v_{2}^{t}=\partial(F(v_{1}^{t}\cup v_{3}))$. Then $F(v_{0} \cup v_{2}) =F(v_{0} \cup v_{2}^{t})$, in particular $v_{1}^{t}=\partial(F(v_{0}\cup v_{2}^{t}))$. 
\end{lemma}

The point of the above statement is that tightness is preserved at time $1$--point after tightening at time $2$--point. Therefore, we can iterate the tightening process inductively on multigeodesics of any length. Hence, a tight geodesic exists between any two vertices in $C(S)$. Furthermore, Masur--Minsky also showed that there are only finitely many tight geodesics between any two vertices in $C(S)$.

Bowditch found a stronger finiteness statement on tight geodesics. First we define

\begin{definition}
Let $a,b\in C(S)$, $A,B\subseteq C(S)$, and $r>0$. 
\begin{itemize}
\item Let $\L(a,b)$ and $\L_{T}(a,b)$ be the set of all geodesics and all tight geodesics between $a$ and $b$ respectively.
\item Let $\d G^{T}(a,b):=\{v\in C(S)| v\in g\in \L_{T}(a,b)\}$, $ G^{T}(A,B):=\cup_{a\in A, b\in B} G^{T}(a,b)$, and $G^{T}(a,b;r):=G^{T}(N_{r}(a), N_{r}(b))$ where $N_{r}(a)$ and $N_{r}(b)$ is a radius $r$--ball around $a$ and $b$ respectively. 
\end{itemize}
\end{definition}

Here, wecall that the curve complex is $\delta$--hyperbolic, which is originally due to Masur--Minsky \cite{MM1}. Indeed, the hyperbolicity constant has been shown to be uniform by Aougab \cite{AOU}, Bowditch \cite{BO4}, Clay--Rafi--Schleimer \cite{CRS}, and Hensel--Przytycki--Webb \cite{HPW}.

The following result was originally due to Bowditch \cite{BO2} without computable bounds. Recently, Webb explicitly computed the bounds \cite{WEB2}.
 
\begin{theorem}[\cite{BO2},  \cite{WEB2}]\label{webb's}
Suppose $\x(S)\geq 2$. Let $a,b\in C(S)$, $r\geq 0$, $\delta$ be the hyperbolicity constant. There exists $K$ depending only on $\x(S)$ such that the following holds:
\begin{enumerate}
\item  If $c\in g\in \L(a,b)$, then $|G^{T}(a,b)\cap N_{\delta}(c)|\leq K.$

\item If $c\in g\in \L(a,b)$ and $c \notin N_{r+j}(a)\cup N_{r+j}(b)$, then $|G^{T}(a,b;r)\cap N_{2\delta}(c)|\leq K.$ 
\end{enumerate}
\end{theorem}
\begin{remark}\label{WEBS}
Webb showed that $j$ can be taken to be $10\delta +1$, that $K$ grows exponentially with $\x(S)$, and that his exponential growth rate is sharp.
\end{remark}

\subsection{Weak tight geodesics}\label{mol}
Weak tight geodesics were introduced in \cite{W2} as an application of his work, which gives a way to see the curve graphs as uniformly locally finite graphs, see Theorem \ref{br}. The definition of weak tight geodesics, Definition \ref{dokoni}, derives from a special behavior of tight geodesics under Masur--Minsky's subsurface projections, Lemma \ref{tightp}. 
First, we briefly recall subsurface projections, see \cite{MM2} for more detail discussions.
\begin{definition}[\cite{MM2}]
Let $Z$ be a subsurface of $S$. Subsurface projection is a set map $\pi_{Z}:C(S)\longrightarrow C(Z)$ defined as follows. Let $x\in C(S).$
\begin{itemize}
\item If $Z$ is not an annulus, then $\pi_{Z}(x)$ is the set of curves in $Z$ obtained by taking the boundary components of $R(a\cup \partial(Z))$ for all $a\in \{x \cap Z\}$. 
\item If $Z$ is an annulus then we first take the annular cover of $S$ which corresponds to $Z$ and compactly the cover with $\partial(\mathbb{H}^{2})$. We denote the cover by $S^{Z}$. We define the annular curve graph of $Z$ using $S^{Z}$; the vertices are the set of isotopy classes of arcs which connect two boundary components of $S^{Z}$, here the isotopy is relative to $\partial (S^{Z} )$ pointwise. We put the edge between two vertices if they can be disjoint in the interior of $S^{Z}$. Lastly, $\pi_{Z}(x)$ is the set of arcs obtained as the preimage of $x$ via the covering map $S^{Z} \longrightarrow S$.

\item Let $A,B \subseteq C(S)$. For both types of projections, we define $ \displaystyle \pi_{Z}(A):=\cup_{a\in A}\pi_{Z}(a)$ and define $d_{Z}(A,B):=\max_{a\in \pi_{Z}(A), b\in \pi_{Z}(B)}d_{Z}(a,b) .$  
\end{itemize}
\end{definition}
The following theorem is called the Bounded Geodesic Image Theorem, which is originally due to Masur--Minsky and recently due to Webb where he explicitly computes its bound which depends only on the hyperbolicity constant. Namely, the bound is uniform. 

\begin{theorem}[\cite{MM2}, \cite{WEB1}]\label{BGIT}
There exists $M$ such that the following holds for all multigeodesics $\{v_{i}\}_{0}^{n}$ in $C(S)$ and $Z\subsetneq S$. If $\pi_{Z}(v_{i})\neq \emptyset$ for all $ i $, then $d_{Z} (v_{0},v_{n}) \leq M.$
\end{theorem}

In the rest of this paper, $M$ will denote the constant given by Theorem \ref{BGIT}. Note that $M\leq 100$ \cite{WEB1}.

Now, we observe the following special property of tight geodesics. 

\begin{proposition}[\cite{W2}]\label{tightp}
Suppose $\x(S)\geq 1.$ Let $x,y \in C(S)$ and $g \in \L_{T}(x,y)$. The following holds for all $v\in g$ and $Z\subsetneq S$. If $\pi_{Z}(v)\neq \emptyset$, then $d_{Z}(x,v)\leq M \text{ or }d_{Z}(v,y)\leq M.$
\end{proposition}

The above proposition naturally motivates the following definition by varying $M.$

\begin{definition}[\cite{W2}]\label{dokoni}
Suppose $\x(S)\geq 1.$ Let $x,y \in C(S)$ such that $d_{S}(x,y)>2$. We say a geodesic $g\in \L(x,y)$ is a $D$--weakly tight geodesic if the following holds for all $v\in g$ and $Z\subsetneq S$. If $\pi_{Z}(v)\neq \emptyset$, then $d_{Z}(x,v)\leq D \text{ or }d_{Z}(v,y)\leq D.$
\end{definition}

Clearly, tight geodesics are $M$--weakly tight geodesics. In \cite{W2}, Theorem \ref{webb's} was proved in the setting of weak tight geodesics as a corollary of Theorem \ref{br}, see Theorem \ref{weaktight}. (We will discuss the implication from Theorem \ref{br} to Theorem \ref{weaktight} in Lemma \ref{imply}.) First we define 
\begin{definition}
Let $a,b\in C(S)$, $A,B\subseteq C(S)$, and $r>0$. 
\begin{itemize}
\item Let $\L_{WT}^{D}(a,b)$ be the set of all $D$--weakly tight geodesics between $a$ and $b$. 
\item Let $\d G^{D}(a,b):=\{v\in C(S)| v\in g\in \L_{WT}^{D}(a,b)\}$, $ G^{D}(A,B):=\cup_{a\in A, b\in B} G^{D}(a,b)$, and $G^{D}(a,b;r):=G^{D}(N_{r}(a), N_{r}(b))$ where $N_{r}(a)$ and $N_{r}(b)$ is a radius $r$--ball around $a$ and $b$ respectively. 
\end{itemize}
\end{definition}

We observe 
\begin{theorem}[\cite{W2}]\label{weaktight}
\it{Suppose $\x(S)\geq1$. Let $a,b\in C(S)$, $r\geq 0$, $\delta$ be the hyperbolicity constant, $D\geq M$, and $N_{S}$ be from Theorem \ref{br}}.
\begin{enumerate}
\item \it{If $c\in g\in \L(a,b)$, then $ |G^{D}(a,b)\cap N_{\delta}(c)| \leq N_{S}(2D,3).$}

\item \it{If $c\in g\in \L(a,b)$ and $c \notin N_{r+j}(a)\cup N_{r+j}(b)$, then $|G^{D}(a,b;r)\cap N_{2\delta}(c)|\leq N_{S}(2(D+M),3).$}
\end{enumerate}
\end{theorem}

\begin{remark}
We can take $j$ to be $3\delta +2$, also our proof works in the case of $\x(S)=1$. However, in the case of tight geodesics, which is understood as $M$--weakly tight geodesics in Theorem \ref{weaktight}, our bounds are weaker than Webb's sharp bounds discussed in Remark \ref{WEBS}. Even though the bounds on Theorem \ref{weaktight}, which are directly given by Theorem \ref{br}, are not expected to be sharp, the gap between Webb's bounds and our bounds was the original motivation for the author to investigate the existence of $M$--weakly tight geodesics that are not tight. In this paper, we confirm this existence by introducing \projection geodesics. However, such existence of \projection geodesics still do not fill the gap between the bounds completely.  
\end{remark}

A special case of Theorem \ref{br} plays a key role in the study of weak tight geodesics, so we briefly discuss it here, emphasizing on its relation to uniformly locally finite graphs and its implication to Theorem \ref{weaktight}.

The following property is equivalent to graphs being uniformly locally finite. (We assume that a graph is path--connected and its diameter is infinite with the graph metric.)
\begin{definition}[\cite{W2}]\label{lfp}
Let $X$ be a graph with the graph metric $d_{X}$. We say $X$ satisfies the uniform local finiteness property if there exists a computable $N_{X}(l,k)$ for any $l>0$ and $k>1$ such that the following holds. If $A\subseteq X$ such that $|A|>N_{X}(l,k)$, then there exists $A'\subseteq A$ such that $|A'|= k$ and $d_{X}(x,y)>l \text{ for all } x,y\in A'.$
\end{definition}
We abbreviate the above property by $\U$. We observe
\begin{proposition}[\cite{W2}]\label{LK}
A graph is uniformly locally finite. $\Longleftrightarrow$ A graph satisfies $\U$.
\end{proposition}

The main result in \cite{W2} was to show the curve graphs satisfy $\U$ via Masur--Minsky's subsurface projections:

\begin{theorem}[\cite{W2}]\label{br}
Suppose $\x(S)\geq 1$. There exists a computable $N_{S}(l,k)$ for any $l>0$ and $k>1$ such that the following holds. If $A\subseteq C(S)$ such that $|A|>N_{S}(l,k)$, then there exists $A'\subseteq A$ and $Z\subseteq S$ such that $|A'|= k$ and $d_{Z}(x,y)>l$ for all $x,y\in A'$.
\end{theorem}

Now we briefly explain how a special case of Theorem \ref{br}, i.e. when $k=3$, implies the first statement of Theorem \ref{weaktight}. (The implication to the second statement requires more work, but it is straightforward after observing the implication to the first.) 
\begin{lemma}\label{imply}
Theorem \ref{br} implies Theorem \ref{weaktight}.
\end{lemma}
\begin{proof}
We know, by the definition of weak tight geodesics, that any element in $G^{D}(a,b)\cap N_{\delta}(c)$ is $D$--close to $x$ or $y$ after projecting to a subsurface. Therefore, we never have $3$ elements in $G^{D}(a,b)\cap N_{\delta}(c)$ that project to a same subsurface where their projection distances are mutually large (more than $2D$). Clearly, the diameter of $G^{D}(a,b)\cap N_{\delta}(c)$ is bounded by $2\delta\leq 2D$ in $C(S)$. We are done by Theorem \ref{br} when $k=3$.
\end{proof}

\section{Results}\label{OUTLINE}
First, we note that 
\begin{lemma}
Every geodesic is a $D$--weakly tight geodesic for some $D$. 
\end{lemma}
\begin{proof}
The statement follows from the fact that a pair of curves can project largely to only finitely many subsurfaces, see \cite{BBF} and \cite{MM2}.
\end{proof}

Now, we start with the following proposition, which directly follows from Definition \ref{dokoni} combining with Proposition \ref{tightp}.

\begin{proposition}[\cite{W2}]\label{Miyake2}
Let $a,b \in C(S)$. Let $n,m\in \mathbb{N}$ such that $n\leq m$. We have
\begin{itemize}
\item $\L_{WT}^{n}(a,b) \subseteq   \L_{WT}^{m}(a,b) .$
\item If $\x(S)=1,$ then $\L_{T}(a,b) = \L_{WT}^{M}(a,b) = \L(a,b).$
\item If $\x(S)>1,$ then $\L_{T}(a,b) \subseteq    \L_{WT}^{M}(a,b) \subseteq  \L(a,b).$
\end{itemize}
\end{proposition}

As we can see in Proposition \ref{Miyake2}, weak tight geodesics in the case of $\x(S)=1$ should not be considered because there is no ``gap''. Therefore, for the rest of this paper, we consider the case when $\x(S)>1$. See Figure \ref{Figk}.

\begin{figure}[h]
\centering
  \includegraphics[width=60mm]{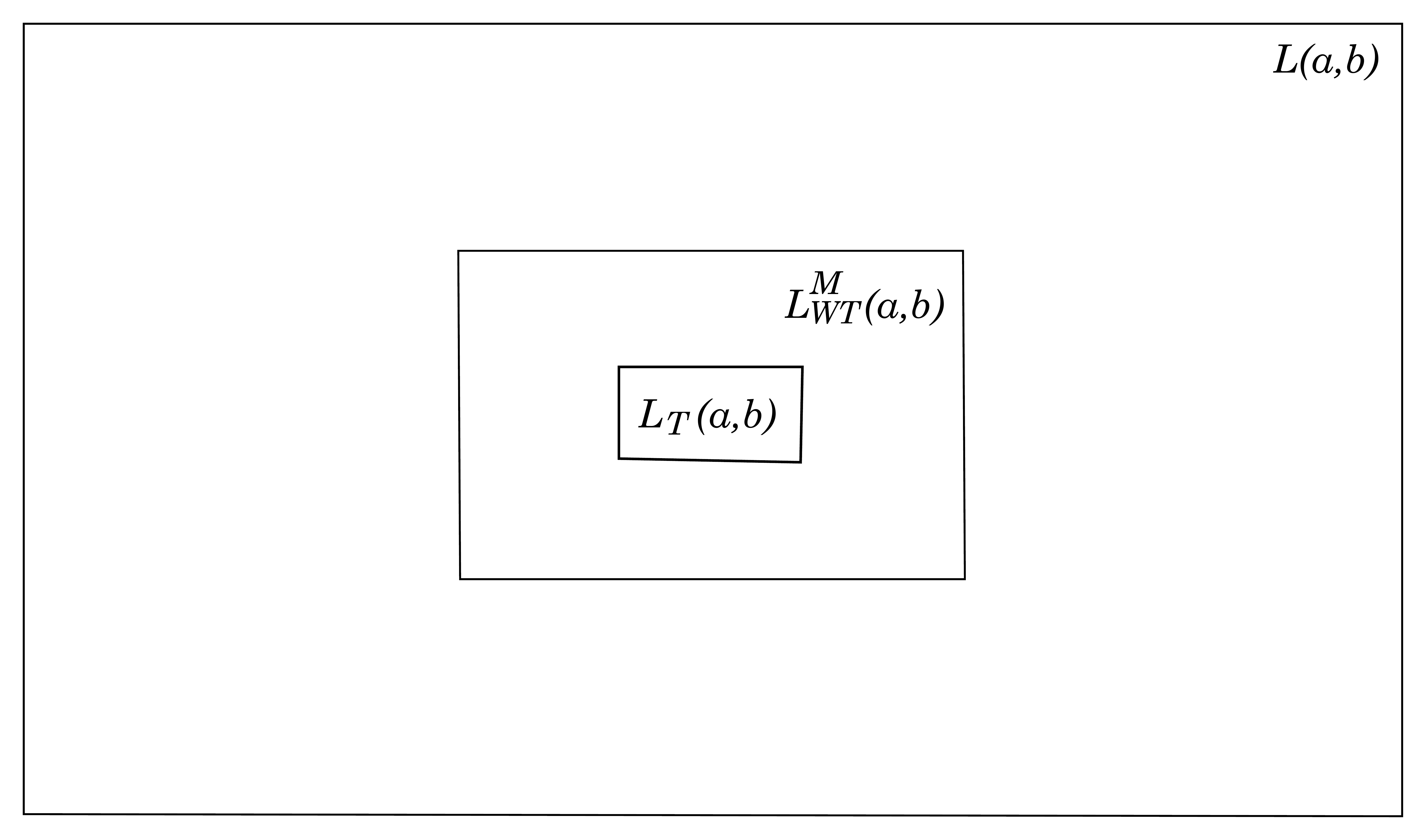}
 \caption{$\L_{T}(a,b) \subseteq    \L_{WT}^{M}(a,b) \subseteq  \L(a,b).$}
 \label{Figk}
\end{figure}

This paper provides constructions of $D$--weakly tight geodesics for a computable $D$ where the constructions are canonical, meaning the constructions apply to any pair of curves. General methods are to start with a geodesic between a given pair of curves and to replace vertices on the geodesic by new vertices which are $D$--close to an endpoint of the geodesic under subsurface projections. However, replacement makes sense only under certain situations. More precisely, we observe 


\begin{proposition}\label{badpair}
Let $x,y\in C(S)$. Let $g=\{v_{i}\}_{i=0}^{d_{S}(x,y)}\in \L(x,y)$ such that the following holds for all $i$: there does not exist a complementary component of $F(v_{i-1}\cup v_{i+1})$ whose complexity is bigger than $0$. Then $g\in \L_{T}(x,y)$. 
\end{proposition} 


Hence, 
\begin{corollary}\label{badpair2}
Let $x,y\in C(S)$. If $\{\L(x,y) \setminus \L_{T}(x,y) \}\neq \emptyset$ then there exists a geodesic $g=\{v_{i}\}_{i=0}^{d_{S}(x,y)}\in \L(x,y)$ such that the following holds for some $i$: there exists a complementary component of $F(v_{i-1}\cup v_{i+1})$ whose complexity is bigger than $0$.
\end{corollary}

If a pair of curves has a non--tight geodesic between them, then the constructions of weak tight geodesics are meaningful: 

\begin{proposition}\label{prop}
Let $x,y\in C(S)$ such that $d_{S}(x,y)>2$ and that $\{ \L(x,y) \setminus \L_{T}(x,y)\}\neq \emptyset$. For all $n\in \mathbb{N}$, $ \{ \L(x,y)\setminus \L_{WT}^{n}(x,y)\}\neq \emptyset$.
\end{proposition}
\begin{proof}
By Corollary \ref{badpair2}, there exists $g=\{v_{i}\}_{i=0}^{d_{S}(x,y)}\in \L(x,y)$ such that the following holds for some $i$: there exists a complementary component of $F(v_{i-1}\cup v_{i+1})$ whose complexity is bigger than $0$. Let $Z$ denote the component. In $C(Z)$, we can pick  a curve which is more than $n$--away from both $\pi_{Z}(x)$ and $\pi_{Z}(y)$; replacing $v_{i}$ by the curve yields a desired geodesic that lives in $\{ \L(x,y)\setminus \L_{WT}^{n}(x,y)\}$. 
\end{proof}

In this paper, we study the following questions, some of which are dual to Proposition \ref{prop} in some sense and some of which were asked in \cite{W2}. 
\begin{question}\label{ns}
Q1) \& Q3) regard constructions and Q2) \& Q4) regard gaps. Let $a,b \in C(S)$. 
\begin{itemize}
\item Q1) In Proposition \ref{Miyake2}, we observed that $\L_{T}(a,b) \subseteq    \L_{WT}^{M}(a,b)$, but is there any other classes of geodesics which live in $\L_{WT}^{M}(a,b)$ with canonical constructions? 

In $\S \ref{PPRR}$, we introduce \Projection geodesics which live in $\L_{WT}^{M}(a,b)$. 

\item Q2) Furthermore, $\{\L_{WT}^{M}(a,b)\setminus \L_{T}(a,b)\}\neq \emptyset$?  

This question will be answered in $\S \ref{bpt}$ when we show that \Projection geodesics are different from tight geodesics. While tight geodesics are contained in \Projection geodesics, there are many \Projection geodesics that are not tight. 

\item Q3) As an extension of Q1), is there other classes of geodesic which live in $\L_{WT}^{D}(x,y)$ for $D>M$ with canonical constructions?

We study this question in the end of $\S \ref{PPRR}$ and Appendix \ref{PPRR2} by extending the construction used to study Q1). 

\item Q4) In Proposition \ref{Miyake2}, we observed that $\L_{WT}^{n}(a,b) \subseteq  \L_{WT}^{m}(a,b)$ for all $n,m\in \mathbb{N}$ such that $n< m$, but how narrow can we take a pair $n,m$ to be so that we can still guarantee that $\{\L_{WT}^{m}(a,b) \setminus \L_{WT}^{n}(a,b)\}\neq \emptyset$? 

We study this question in Appendix \ref{NARROW}. First, we point out some difficulties in studying Q4):
denote the curve picked in the proof of Proposition \ref{prop} for the replacement of $v_{i}$ by $\gamma$. In fact we could have picked $\gamma$ so that $d_{Z}(x,\gamma)>n$ and $d_{Z}(\gamma, y)>n$, but $d_{Z}(x,\gamma)\leq m$ or $d_{Z}(\gamma, y)\leq m$ for all $m$ such that $m>n$. However, a new geodesic with the replacement by $\gamma$ may not be $m$--weakly tight at time $i$--point since there could exist a subsurface $W\subsetneq Z$ such that $d_{W}(x,\gamma)>m$ and $d_{W}(\gamma, y)>m$. 
Nevertheless, the replacement works effectively if $Z$ is a complementary component of $F(v_{i-1}\cup v_{i+1})$ when $i=1$ or $i=d_{S}(x,y)-1$. Specifically focusing on this case, in Appendix \ref{NARROW}, we will discuss the construction of weak tight geodesics which live in $\{\L_{WT}^{m}(x,y) \setminus \L_{WT}^{n}(x,y)\}$ for all $n,m\in \mathbb{N}_{\geq M}$ such that $m-n>2$.
\end{itemize}
\end{question}


We outline the plan of this paper. 

In $\S \ref{PPRR}$, we introduce \projection geodesics, denoted $\mathcal{L}_{\EE}(a,b)$, and prove they are contained in the class of $(M+2)$--weakly tight geodesics. We note that a special subclass of \projection geodesics, denoted $\mathcal{L}_{\PP}(a,b)$, are contained in the class of $M$--weakly tight geodesics. For simplicity, we discuss the construction of $\mathcal{L}_{\PP}(a,b)$ here. Roughly speaking, starting with a geodesic, we replace the vertices of the geodesic by the following method. Fix a vertex of the geodesic, and replace the vertex by the curves obtained by the following two procedures:
taking the boundary components of the subsurface filled by two curves that are adjacent to the fixed vertex, and projecting the endpoints of the geodesic to the complementary components of the above subsurface.
See Definition \ref{projection}. Note that the first procedure is exactly tightening procedure, but we get more curves by the second procedure, that is, projecting procedure. 
The key point is that the curves obtained by this method will be $M$--close to an endpoint of the given geodesic under any subsurface projections, see Lemma \ref{weekend}. We iterate this replacement process to every vertex of the given geodesic, so that we have a $M$--weakly tight geodesic at the final stage, see Example \ref{EEXX}. Since we get curves for the replacement by tightening and by projecting, we call this geodesic a \emph{\projection geodesic.} An actual definition of \projection geodesics, $\mathcal{L}_{\EE}(a,b)$, is a generalization of $\mathcal{L}_{\PP}(a,b)$; instead of using only the endpoints of the given geodesic for projecting procedure, we use all curves contained in the balls of certain radii centered at the endpoints.

It is straightforward to see that \projection geodesics contain tight geodesics because of tightening procedure. Moreover, at each replacement, we generally get more curves because of projecting procedure, hence it is natural to believe that we generally get \projection geodesics that are not tight by doing both procedures on the given geodesic, yet those geodesics could arise as tight geodesics by doing only tightening procedure on a different geodesic connecting the same given endpoints. 
Therefore, in $\S\ref{bpt}$, in order to confirm that \projection geodesics are different from tight geodesics, we provide examples of a pair of curves where there are \projection geodesics between them that are not tight, see Example \ref{lox} and Example \ref{lox3}. Such geodesics live in the gap between $M$--weakly tight geodesics and tight geodesics. See Figure \ref{Figk1}.
\begin{figure}[h]
\centering
  \includegraphics[width=40mm]{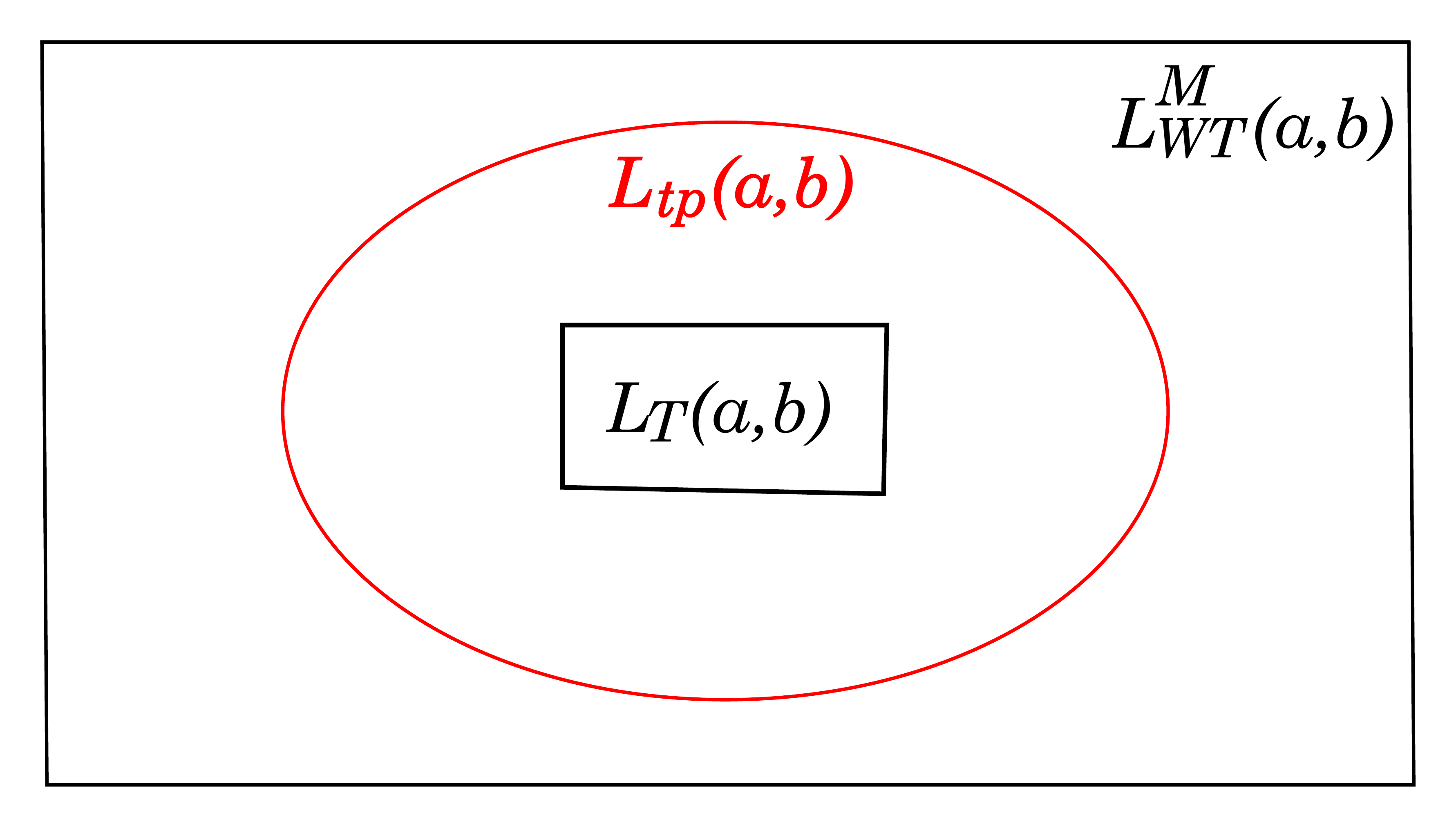}
 \caption{$\mathcal{L}_{T}(a,b)\subseteq \mathcal{L}_{\PP}(a,b) \subseteq \mathcal{L}_{WT}^{M}(a,b)$.}
 \label{Figk1}
\end{figure}
Example \ref{lox} and Example \ref{lox3} discuss pairs of curves of distance $3$, so the difference between tight geodesics and \projection geodesics may not be very visible. For this, we find sequences of pairs of curves whose distances go to infinity where tight geodesics and \projection geodesics between them behave very differently in the limit, see Lemma \ref{hc1} and Lemma \ref{hc2}.  

In Appendix \ref{PPRR2}, we give other examples of weak tight geodesics; by using the curves obtained for the replacement procedures discussed in $\S \ref{PPRR}$ as base curves, we add more curves that intersect every base curve at most $n$ times, and use them for replacement. Naturally, geodesics with this replacement procedure will live in a larger class of weak tight geodesics, see Corollary \ref{stbk}.

In Appendix \ref{NARROW}, we focus on special pairs of curves where complementary components with high complexity discussed in Corollary \ref{badpair2} occur at time $1$--point or time $(d_{S}(x,y)-1)$--point. In this case, we construct a geodesic which lives in a narrow gap between two classes of weak tight geodesics, see Corollary \ref{sincl2}.

\section{\Projection geodesics}\label{PPRR}
The goal of this section is to introduce \emph{\projection geodesics}. 
The existence of \projection geodesics will easily follow from the definition. 

First we define the following.
\begin{definition}\label{projection}
Suppose $A\subseteq S$. We let $A^{c}$ denote the set of complementary components of $A$ in $S$. Furthermore, we let $A^{nac}$ denote the set of non--annular complementary components of $A$ in $S.$ We define $\Proj_{A}:C(S)\longrightarrow C(S),$ where $$\Proj_{A}(x):=\bigg(\bigcup_{Y\in A^{nac}} \pi_{Y}(x) \bigg) \bigcup \partial(A).$$ See Figure \ref{Fig1}.
We note that $\Proj_{A}(x)$ are contained in $A^{c}$, that $\Proj_{A}(x)\neq \emptyset$ as far as $A\subsetneq S$, and that $\{\Proj_{A}(x) \setminus  \partial(A)\} \neq \emptyset$ as far as $x$ essentially intersects with a non--pants component of $A^{nac}$. Let $C\subseteq C(S)$. We define $\Proj_{A}(C):=\cup_{x\in C}\Proj_{A}(x).$ 
\begin{figure}[h]
\centering
  \includegraphics[width=100mm]{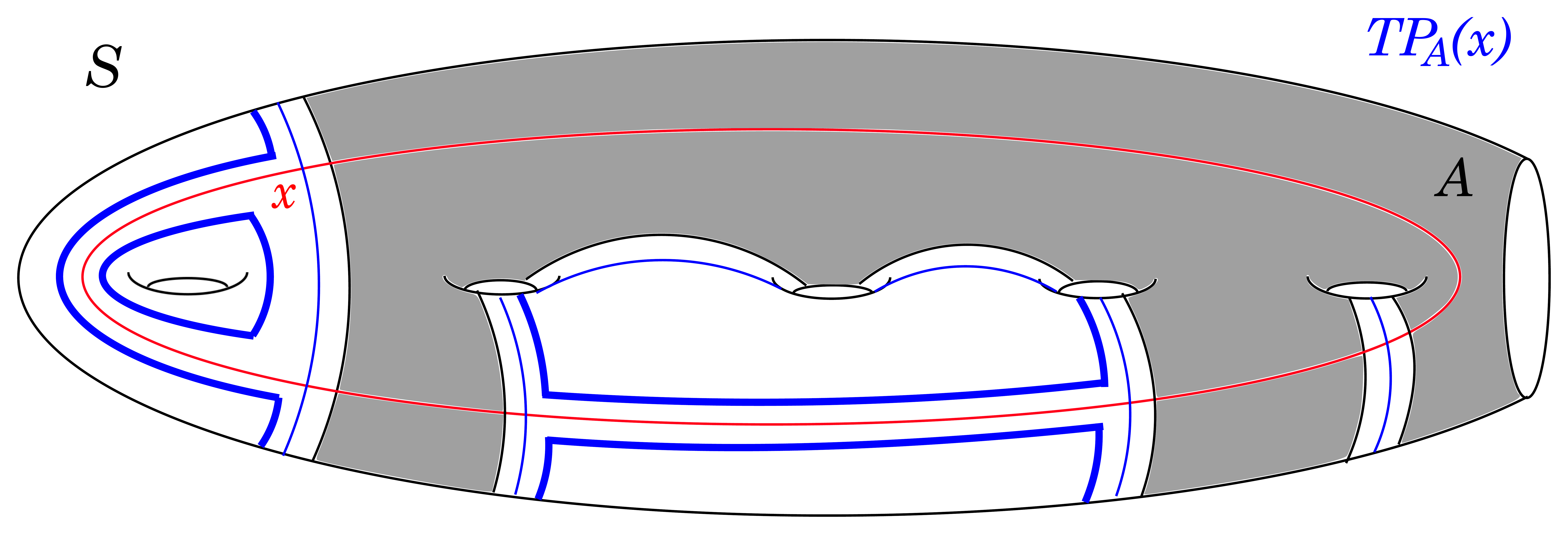}
 \caption{$\Proj_{A}(x)$ when $A$ is the shaded region. Bold blue curves and regular blue curves in $A^{c}$ are the elements of $\Proj_{A}(x)$ arising from  projections and from $\partial(A)$ respectively.}
 \label{Fig1}
\end{figure}

\end{definition}

The following lemma is the key. 

\begin{lemma}\label{weekend}
Let $x,y\in C(S)$ such that $d_{S}(x,y)>2$. Let $\{ x=v_{0}, v_{1}, \cdots ,v_{d_{S}(x,y)}=y\}$ be a multigeodesic connecting $x$ and $y$. The following holds for all $i$:  
\begin{itemize}
\item Replacing $v_{i}$ by any multicurve which is a subset of $\Proj_{F(v_{i-1}\cup v_{i+1})}(x \cup y)$ gives a new (possibly same) multigeodesic.
\item If $\c \in  \Proj_{F(v_{i-1}\cup v_{i+1})}(x \cup y)$ and $\pi_{Z}(\gamma)\neq \emptyset$ where $Z\subsetneq S$, then $d_{Z}(x,\gamma)\leq M \text{ or }d_{Z}(\gamma,y)\leq M.$
\end{itemize}
\end{lemma} 
\begin{proof}
The first statement is obvious. For the second statement, we assume $\c \in  \Proj_{F(v_{i-1}\cup v_{i+1})}(x)$. We show $$d_{Z}(x,\gamma)\leq M.$$ (If $\c \in  \Proj_{F(v_{i-1}\cup v_{i+1})}(y)$ then $d_{Z}(\gamma,y)\leq M.$)

Case 1: Suppose $\pi_{Z}(v_{i-1})= \emptyset$ and $\pi_{Z}(v_{i+1})= \emptyset$. Then $Z\subseteq W \in (F(v_{i-1}\cup v_{i+1}))^{c}$, in fact $Z \subseteq W \in (F(v_{i-1}\cup v_{i+1}))^{nac}$ since $\pi_{Z}(\c) \neq \emptyset$. Also, we observe $\c \notin \partial(W) \subseteq \partial ((F(v_{i-1}\cup v_{i+1})))$ since $\pi_{Z}(v_{i-1})= \emptyset$ and $\pi_{Z}(v_{i+1})= \emptyset$. Therefore, $$\gamma \in \pi_{W}(x).$$

If $Z$ is not an annulus, then, by the definition of subsurface projections, we have $$\pi_{Z}(\pi_{W}(x))=\pi_{Z}(x).$$ Since $\gamma \in \pi_{W}(x)$, we have $$\pi_{Z}(\c)\subseteq \pi_{Z}(\pi_{W}(x)) = \pi_{Z}(x) \Longrightarrow d_{Z}(x,\c)\leq 2.$$ 

If $Z$ is an annulus, then we do not have $\pi_{Z}(\pi_{W}(x))=\pi_{Z}(x)$ in general. However, we show that $$d_{Z}(\gamma, x)\leq 2.$$ Let $A^{Z}$ denote the annular cover of $S$ which corresponds to $Z$. Let $a \in \{x \cap W\}$ be an arc in $W$ such that $\gamma \in \pi_{W}(a)\subseteq \pi_{W}(x)$. Note that $\pi_{Z}(\partial(W))=\emptyset$ since $i(\partial(Z),\partial(W))=0$, in other words, $\partial(W)$ lifts to arcs whose endpoints are contained in a single boundary component of $A^{Z}.$
Since $a \in \{x \cap W\}$, the lift of $a$ is contained in some of the lift of $x$, say $x'$. Furthermore, the endpoints of the lift of $a$ are contained in the lift of $\partial(W)$ in $A^{Z}$. See Figure \ref{annulus}. 
Now, since $i(\gamma, a)=i(\gamma, \partial(W))=0$, neither the lift of $a$ nor the lift of $\partial(W)$ intersects with the lift of $\gamma$ in $A^{Z}$; we have $$d_{Z}(\gamma, x')\leq 1.$$ Lastly, since the diameter of $\pi_{Z}(x)$ is bounded by $1$ in $C(Z)$, we have $$d_{Z}(\gamma, x)\leq 2.$$
\begin{figure}[h]
\centering
  \includegraphics[width=50mm]{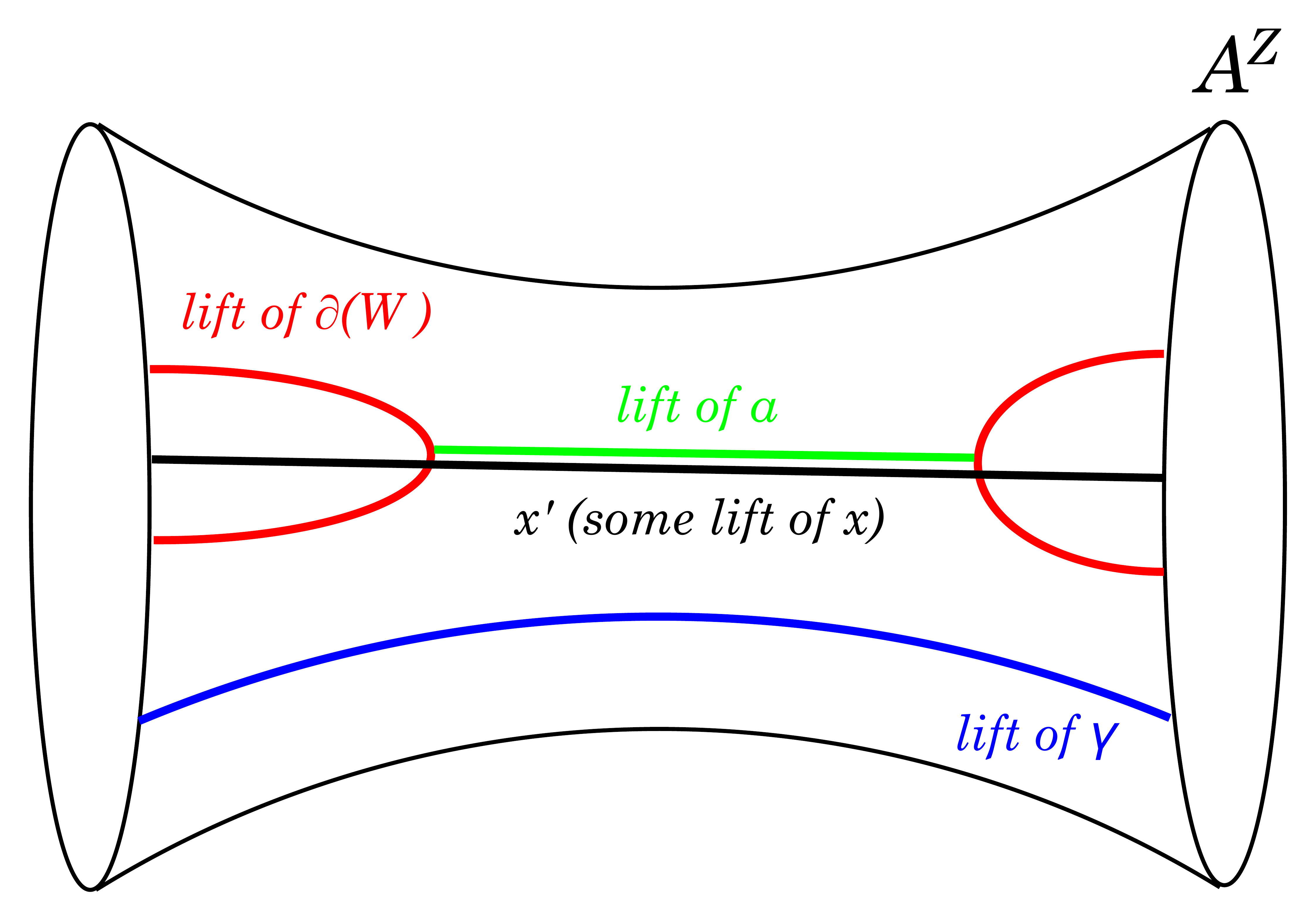}
 \caption{$d_{Z}(\gamma, x')\leq 1.$}
 \label{annulus}
\end{figure}

Case 2: Suppose $\pi_{Z}(v_{i-1})\neq \emptyset$ or $\pi_{Z}(v_{i+1})\neq \emptyset$. If $\pi_{Z}(v_{i+1})= \emptyset$, then we have $\pi_{Z}(v_{k})\neq \emptyset$ for all $k<i-1$ since $v_{k}$ and $v_{i+1}$ fill $S$. Therefore, by Theorem \ref{BGIT}, we have $$d_{Z}(x,\c)\leq  M.$$ 

An analogous argument applies to the case $\pi_{Z}(v_{i-1})= \emptyset$, and also to the case $\pi_{Z}(v_{i-1})\neq  \emptyset$ and $\pi_{Z}(v_{i+1})\neq  \emptyset$. 
\end{proof}

To get a flavor of the definition of \projection geodesics, we first observe the following basic construction of \projection geodesics. Note that it will be a $M$--weakly tight.

\begin{example}\label{EEXX}
Let $x,y\in C(S)$ such that $d_{S}(x,y)>2$. Let $\{x=v_{0}, v_{1}, \cdots ,v_{d_{S}(x,y)}=y\}$ be a multigeodesic connecting $x$ and $y$. First, we replace $v_{1}$ by a multicurve, denoted $v_{1}'$, which arises from $ \Proj_{F(v_{0}\cup v_{2})}(x \cup y)$ and obtain a new multigeodesic. On this new multigeodesic, $\{x=v_{0}, v_{1}', \cdots ,v_{d_{S}(x,y)}=y\}$, we replace $v_{2}$ by a multicurve, denoted $v_{2}'$, which arises from $\Proj_{F(v_{1}' \cup v_{3})}(x \cup y)$. We iterate this replacement procedure consecutively on $i$ until we replace $v_{d_{S}(x,y)-1}$. By Lemma \ref{weekend}, the final multigeodesic is $M$--weakly tight.
\end{example}

We observe, in Example \ref{EEXX}, $\Proj_{F(v_{0}\cup v_{2})}(x \cup y)$ generally contains elements that arise from projections, and these may not span a simplex in $C(S)$. This is different from the case of tight geodesics where we always get a simplex arising from the boundary components of subsurfaces. For a general construction of \projection geodesics, it would be more convenient to work with the notion of \emph{thick geodesics} where we loosen up the definition of multigeodesics given in Definition \ref{d} by not requiring $v_{i}$ to be a simplex thereof. Clearly, geodesics and multigeodesics are thick geodesics. Let $x,y\in C(S)$. We let $\L_{thick}(x,y)$ denote the set of all thick geodesics between $x$ and $y.$ 


It is straightforward to see that Lemma \ref{weekend} holds in the setting of thick geodesics. 
Now, we generalize Lemma \ref{weekend}; while we only used the endpoints of a given thick geodesic, we make use of much more curves around the endpoints. Recall that $N_{r}(c)$ denotes the radius $r$--ball centered at $c\in C(S)$.
\begin{lemma}\label{weaken}
Let $x,y\in C(S)$ such that $d_{S}(x,y)>2$. Let $\{ x=v_{0}, v_{1}, \cdots ,v_{d_{S}(x,y)}=y\}$ be a thick geodesic connecting $x$ and $y$. The following holds for all $i$:  
\begin{itemize}
\item Replacing $v_{i}$ by any subset of $\Proj_{F(v_{i-1}\cup v_{i+1})} \big( N_{i-1}(x) \cup N_{d_{S}(x,y)-i-1}(y) \big) $ gives a new (possibly same) thick geodesic.
\item If $\c \in  \Proj_{F(v_{i-1}\cup v_{i+1})} \big( N_{i-1}(x) \cup N_{d_{S}(x,y)-i-1}(y) \big) $ and $\pi_{Z}(\gamma)\neq \emptyset$ where $Z\subsetneq S$, then $d_{Z}(x,\gamma)\leq M+2 \text{ or }d_{Z}(\gamma,y)\leq M+2.$
\end{itemize}
\end{lemma} 
\begin{proof}
The first statement is obvious. For the second statement, we assume $\c \in  \Proj_{F(v_{i-1}\cup v_{i+1})}(c)$ where $c\in N_{i-1}(x)$.

Case 1: Suppose $\pi_{Z}(v_{i-1})= \emptyset$ and $\pi_{Z}(v_{i+1})= \emptyset$. Then $Z \subseteq W \in (F(v_{i-1}\cup v_{i+1}))^{nac}$. We have $d_{Z}(x,c)\leq M$ by Theorem \ref{BGIT} since $\pi_{Z}(v_{i+1})= \emptyset$, and $d_{Z}(c,\c)\leq 2$ by Lemma \ref{weekend}, so we have $d_{Z}(x,\c)\leq d_{Z}(x, c)+d_{Z}(c,\c)\leq M+2.$

Case 2: Suppose $\pi_{Z}(v_{i-1})\neq \emptyset$ or $\pi_{Z}(v_{i+1})\neq \emptyset$. We have $d_{Z}(x,\gamma)\leq M$ or $d_{Z}(\gamma,y)\leq M.$
\end{proof}


\begin{figure}[h]
\centering
  \includegraphics[width=100mm]{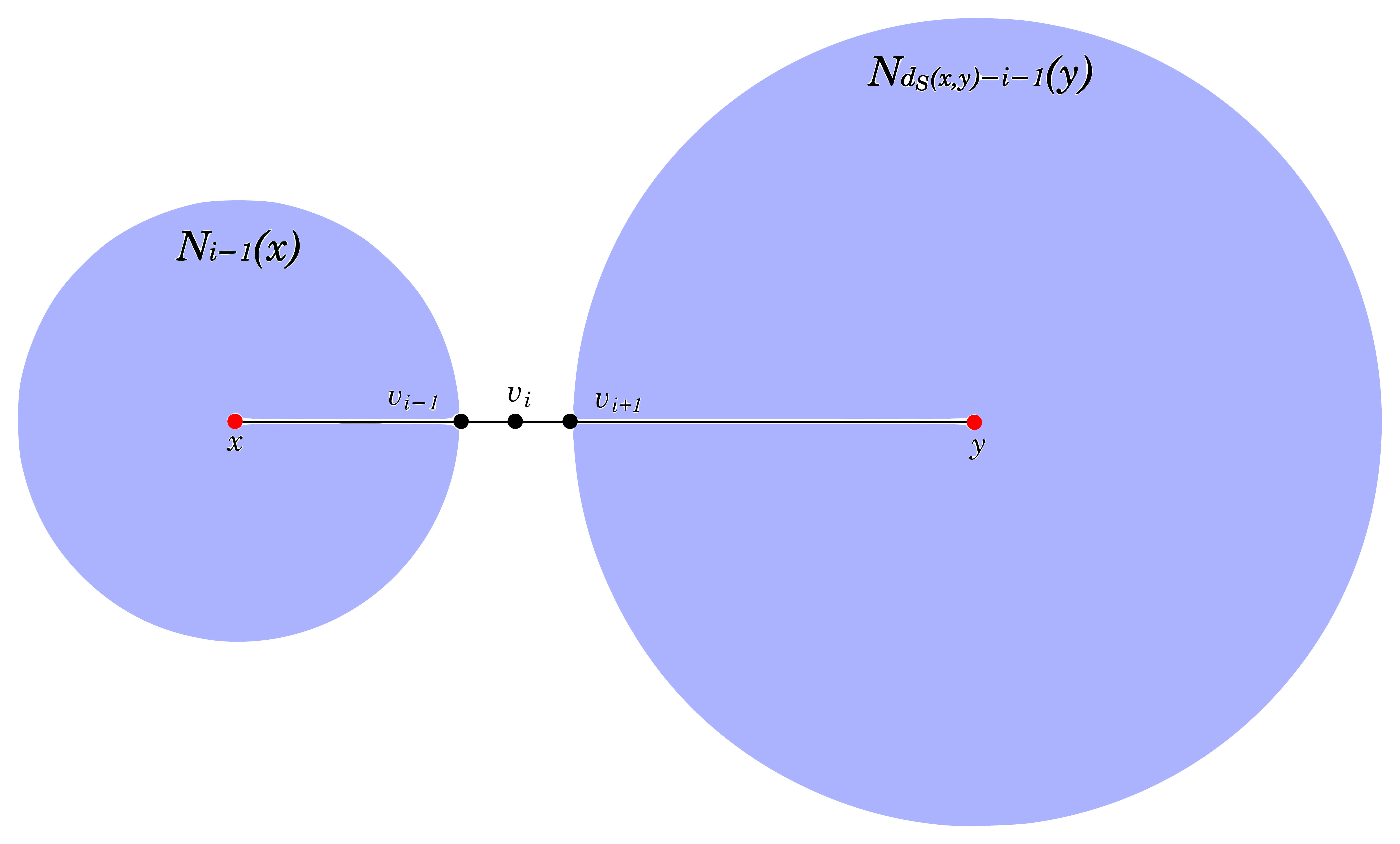}
 \caption{Red curves are used in Lemma \ref{weekend} and blue curves are used in Lemma \ref{weaken} via $\Proj_{F(v_{i-1}\cup v_{i+1})}$. See Figure \ref{Figk10}.}
 \label{Figk2}
\end{figure}

Now, we define \Projection geodesics.
\begin{definition}[\Projection geodesics]\label{pg}
Let $x,y\in C(S)$ such that $d_{S}(x,y)>2$. Let $g=\{x=v_{0}, v_{1}, \cdots ,v_{d_{S}(x,y)}=y\} \in \L_{thick}(x,y)$. We let $\EE_{i}$ denote the replacement procedure on time $i$--point, $$\EE_{i}:\L_{thick}(x,y)\longrightarrow \L_{thick}(x,y),$$ where $\EE_{i}(g)$ is the set of all thick geodesics which can possibly arise from replacing $v_{i}$ by the subsets of $\Proj_{F(v_{i-1}\cup v_{i+1})} \big( N_{i-1}(x) \cup N_{d_{S}(x,y)-i-1}(y) \big)$. If $G\subseteq \L_{thick}(x,y)$ then we define $\EE_{i}(G):=\cup_{g\in G}\EE_{i}(g).$ We define \[\L_{\EE}(x,y):=  \bigg( \EE_{d_{S}(x,y)} \circ \cdots  \circ \EE_{2}  \circ \EE_{1} \big(\L_{thick}(x,y)\big)  \bigg) \cap \L(x,y).  \] We call an element of $\L_{\EE}(x,y)$ a \emph{\projection geodesic}.

As a special case, we alter $\EE_{i}(g)$ by $\PP_{i}(g)$ where $\PP_{i}(g)$ is the set of all thick geodesics which can possibly arise from replacing $v_{i}$ by the subsets of $\Proj_{F(v_{i-1}\cup v_{i+1})}(x \cup y) $. We define \[\L_{\PP}(x,y):=  \bigg( \PP_{d_{S}(x,y)} \circ \cdots  \circ \PP_{2}  \circ \PP_{1} \big(\L_{thick}(x,y)\big)  \bigg) \cap \L(x,y).  \] 

Clearly, we have $$\L_{\PP}(x,y)\subseteq \L_{\EE}(x,y).$$
\end{definition}

One could see that, in Example \ref{EEXX} and in Definition \ref{pg}, we do not need to start the replacement at time $1$--point. We just need to replace every vertex of a given thick geodesic. In particular, instead of using $\EE_{d_{S}(x,y)} \circ \cdots  \circ \EE_{2}  \circ \EE_{1}$, we can use any composition consisting of the elements of $\{\EE_{i}\}_{i=0}^{d_{S}(x,y)}$ such that each $\EE_{i}$ appears at least once, and use it to define \projection geodesics. In fact, one can check that it does not make any difference in the elements of $\L_{\EE}(x,y)$. Hence, even though the index of $\projection_{i}$ depends on the choice of time $0$--point, which can be either $x$ or $y$, the choice is not important in Definition \ref{pg}. 

We observe 
\begin{theorem}\label{M-weak}
Let $x,y\in C(S)$ such that $d_{S}(x,y)>2$. Then 
\begin{itemize}
\item $\L_{\PP}(x,y)\neq \emptyset$ and $\L_{\EE}(x,y)\neq \emptyset.$
\item $\L_{T}(x,y) \subseteq \L_{\PP}(x,y).$
\item $\L_{T}(x,y)\subseteq \L_{\PP}(x,y)\subseteq \L_{WT}^{M}(x,y).$
\item $\L_{T}(x,y)\subseteq \L_{\PP}(x,y)\subseteq  \L_{\EE}(x,y) \subseteq \L_{WT}^{M+2}(x,y).$
\end{itemize}
\end{theorem}
\begin{proof}
The first statement is obvious by Definition \ref{pg}. 
For the second statement, let $g\in \L_{T}(x,y)$ and $G$ be a tight multigeodesic which gives rise to $g$. Then $G\in \PP_{d_{S}(x,y)} \circ \cdots  \circ \PP_{2}  \circ \PP_{1}(G)$, hence $g\in  \L_{\PP}(x,y) .$
The third statement follows with Lemma \ref{weekend}.
The last statement follows with Lemma \ref{weaken}. \begin{figure}[h]
\centering
  \includegraphics[width=50mm]{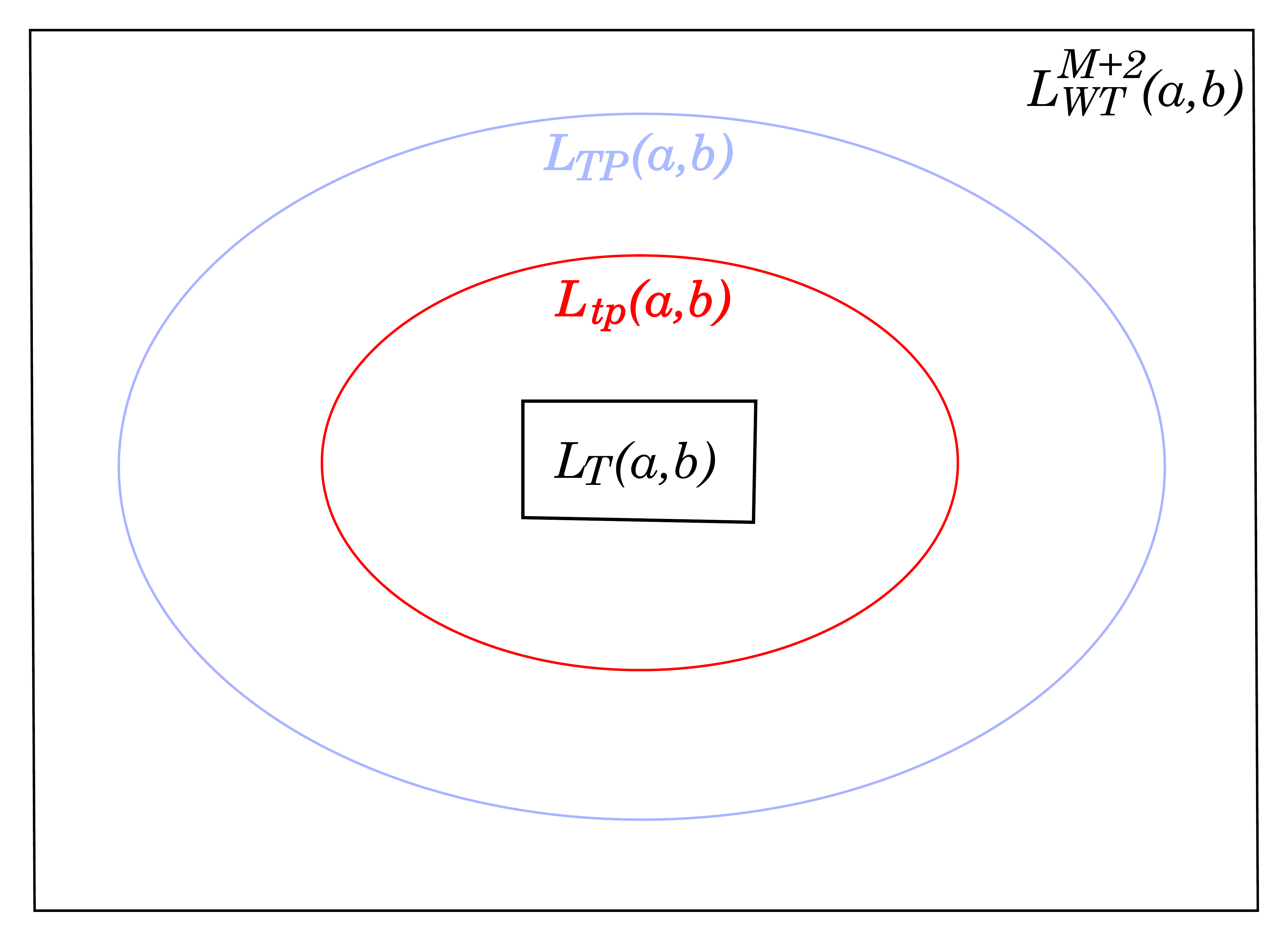}
 \caption{$\mathcal{L}_{T}(a,b)\subseteq \mathcal{L}_{\PP}(a,b) \subseteq \mathcal{L}_{\EE}(a,b) \subseteq \mathcal{L}_{WT}^{M+2}(a,b)$.}
 \label{Figk10}
\end{figure}
\end{proof}

The existence of \projection geodesics is immediate compared to the existence of tight geodesics which requires Lemma \ref{construction}. This is by virtue of the definition of weak tight geodesics. 

\section{Tight geodesics versus \Projection geodesics}\label{bpt}

We have observed so far that \projection geodesics contain tight geodesics, we now give specific pairs of curves where there exist \projection geodesics, which are not tight, between them. 
\begin{example}\label{lox}
Let $a,b\in C(S)$ in Figure \ref{Fig2}. Recall that $b^{c}$ means the complementary component of $b$ in $S$, in this case $b^{c}$ contains one component. We can pick $\gamma \in C(b^{c})$ such that $d_{b^{c}}(a,\gamma) \gg M.$ By Theorem \ref{BGIT} and the fact that $b$ is a non--separating curve, we can conclude that every geodesic between $a$ and $\gamma$ needs to contain $b$; in particular this means that $d_{S}(a,\gamma)=3$ since $d_{S}(a,b)=2$. 
\begin{figure}[h]
\centering
  \includegraphics[width=50mm]{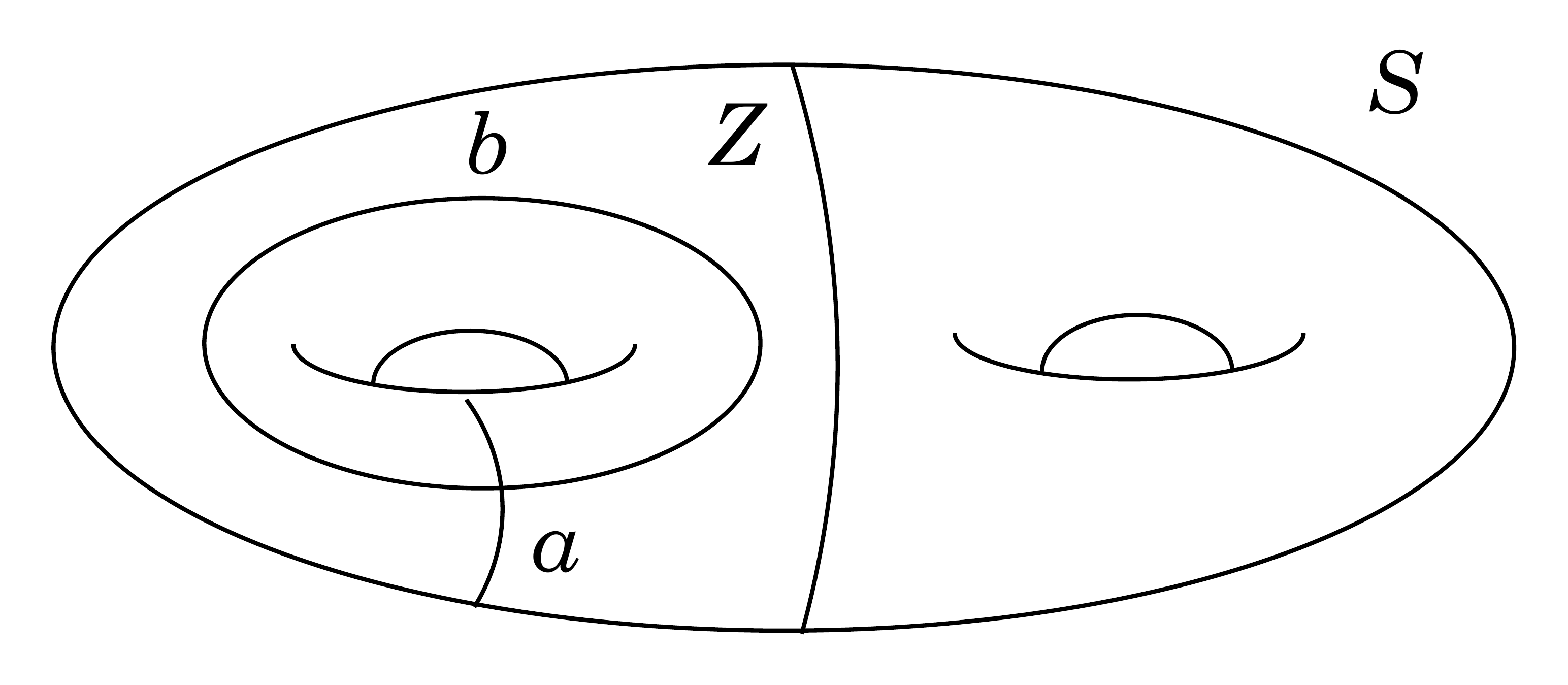}
 \caption{The tight geodesic \{$a,\partial(Z),b,\gamma$\}.}
 \label{Fig2}
\end{figure}
We claim the following.
\begin{enumerate}
\item $|\L_{T}(a,\gamma)|=1$.
\item $| \L_{\PP}(a,\gamma) |=| \Proj_{F(a\cup b)}(  \gamma)  |$.
\end{enumerate}

\emph{Proof of Claim $1$.} Since every geodesic between $a$ and $\gamma$ needs to contain $b$, we conclude that $\{a,\partial(Z),b,\gamma\}$ is the only tight geodesic between $a$ and $\gamma$.\qed

\emph{Proof of Claim $2$.} We have $$ \L_{\PP}(a,\gamma)=\{a,x,b, \gamma | x \in \Proj_{F(a \cup b)}( \gamma ) \}.$$\qed


\end{example}

The following example is a general version of Example \ref{lox}. 
\begin{example}\label{lox3}
Let $a,b_{1}\in C(S)$ and let $Z_{1}\subsetneq Z_{2} \subsetneq \cdots \subsetneq Z_{n-1}\subsetneq  Z_{n}\subsetneq S$ be a sequence of subsurfaces in Figure \ref{Fig3}. 
\begin{figure}[h]
\centering
  \includegraphics[width=100mm]{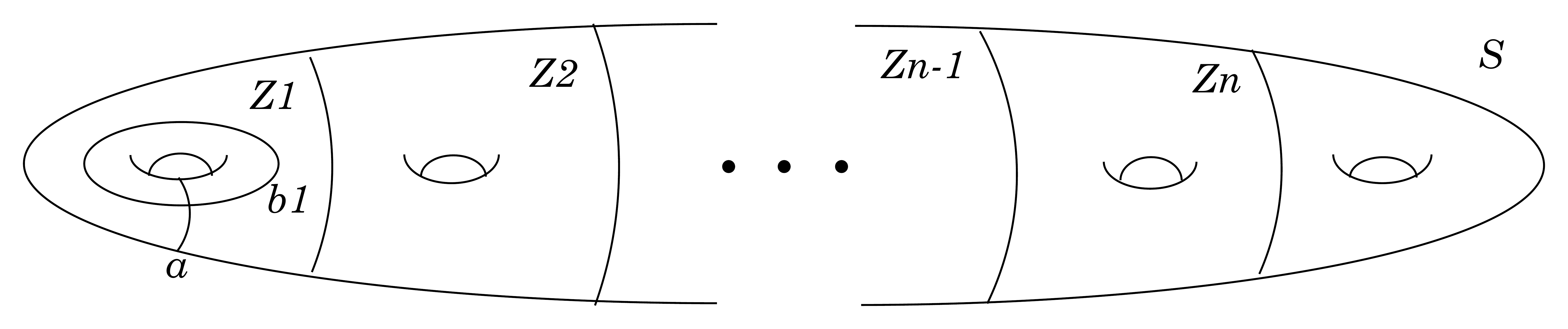}
 \caption{$Z_{1}\subsetneq Z_{2} \subsetneq \cdots \subsetneq Z_{n-1}\subsetneq  Z_{n}\subsetneq S$.}
 \label{Fig3}
\end{figure}
We inductively pick a non--separating curve $b_{i}$ so that 
\begin{itemize}
\item $b_{i}$ fills $Z_{i}$ with $a$
\item $i(b_{i},b_{j})=0$ for all $j< i.$ 
\end{itemize}
For instance, $b_{i}$ can be taken from $C(Z_{i}\setminus \cup_{j< i}b_{j})$ so that $d_{(Z_{i}\setminus \cup_{j< i}b_{j})}(a,b_{i})\gg M$.

Let $Y=\{b_{i}\}_{i=1}^{n}$, which is a simplex. Note that $Y^{c}$ contains only one component since $b_{i}$ is a non--separating curve for all $i$. We pick $\gamma \in C(Y^{c})$ so that $d_{Y^{c}}(a, \gamma)\gg M$, then a geodesic between $a$ and $\gamma$ needs to contain $b_{i}$ for some $i$. Indeed, we can always find a geodesic passing through $b_{i}$ for all $i$, i.e., $\{a,\partial(Z_{i}),b_{i},\gamma \}$.

We claim the following.

\begin{enumerate}
\item $|\L_{T}(a,\gamma)|= \sum_{i=1}^{n}|\cup_{j\geq i} \partial(Z_{j})|= \sum_{i=1}^{n}i$.
\item $|\L_{\PP}(a,\gamma)|= \sum_{i=1}^{n}|\cup_{j\geq i} \Proj_{Z_{j}}(  \gamma)|= \sum_{i=1}^{n} i\cdot | \{ \Proj_{Z_{i}}(  \gamma)\setminus \cup_{j>i} \Proj_{Z_{j}}(  \gamma) \}|$.
\item $|\L_{\PP}(a,\gamma)| \geq 2\cdot | \L_{T}(a,\gamma)|$.

\end{enumerate}

\emph{Proof of Claim $1$.} Possible simplices appearing on tight multigeodesics between $a$ and $\gamma$, which are adjacent to $\gamma$, are formed by $\{b_{i}\}_{i=1}^{n}$; let $\triangle_{i}$ denote a simplex which contains $b_{i}$ as the highest index curve. For instance, $\triangle_{2}$ can be $\{b_{2}\}$ or $\{b_{1},b_{2}\}$. We observe that $F(a,\triangle_{i}) =Z_{i}$ so $\partial(F(a,\triangle_{i}))=\partial(Z_{i})$ for all $i$. 
We prove that $\partial(F(\partial(Z_{i}),\gamma))= \{b_{j}| j\leq i\}$ so that we can conclude that every tight geodesic between $a$ and $\gamma$ is of the following form; $$ \cup_{0<i\leq n}\{a,\partial(Z_{i}),b_{j},\gamma| j\leq i\} =\cup_{0<i\leq n}\{a,\partial(Z_{j}),b_{i},\gamma| j\geq i\}.$$

Now, we prove that $\partial(F(\partial(Z_{i}),\gamma))= \{b_{j}| j\leq i\}$: since $d_{S}(a,\gamma)=3$ and $a$ is contained in $Z_{i}$, we observe that $Z_{i}^{c}\subseteq F(\partial(Z_{i}), \gamma)$, otherwise $d_{S}(a,\gamma)$ would be $2$. Therefore $\partial(F(\partial(Z_{i}),\gamma))\subseteq Z_{i}$. Clearly $b_{j}\subseteq F(\partial(Z_{i}),\gamma)^{c}$ for all $j\leq i$. If $\partial(F(\partial(Z_{i}),\gamma))\neq  \{b_{j}| j\leq i\}$, then there exists $x\in C(Z_{i})$ such that $\pi_{Y^{c}}(x)\neq \emptyset$ so that $$d_{Y^{c}}(\gamma, \partial(Z_{i}))\leq d_{Y^{c}}(\gamma, x)+d_{Y^{c}}( x, \partial(Z_{i})) \leq 2+2,$$ which is a contradiction since $d_{Y^{c}}(\partial(Z_{i}), \gamma)\gg M$ as $d_{Y^{c}}(a, \gamma)\gg M$. \qed

\emph{Proof of Claim $2$.} Every \projection geodesic passing through $b_{i}$ is of the following form: $$ \cup_{0<i\leq n}\{a,c, b_{i}, \gamma | c\in    \cup_{j\geq i} \Proj_{Z_{j}}(  \gamma)\}.$$ \qed 

\emph{Proof of Claim $3$.} See Figure \ref{tightvsBP}.  
\begin{figure}[h]
\centering
  \includegraphics[width=120mm]{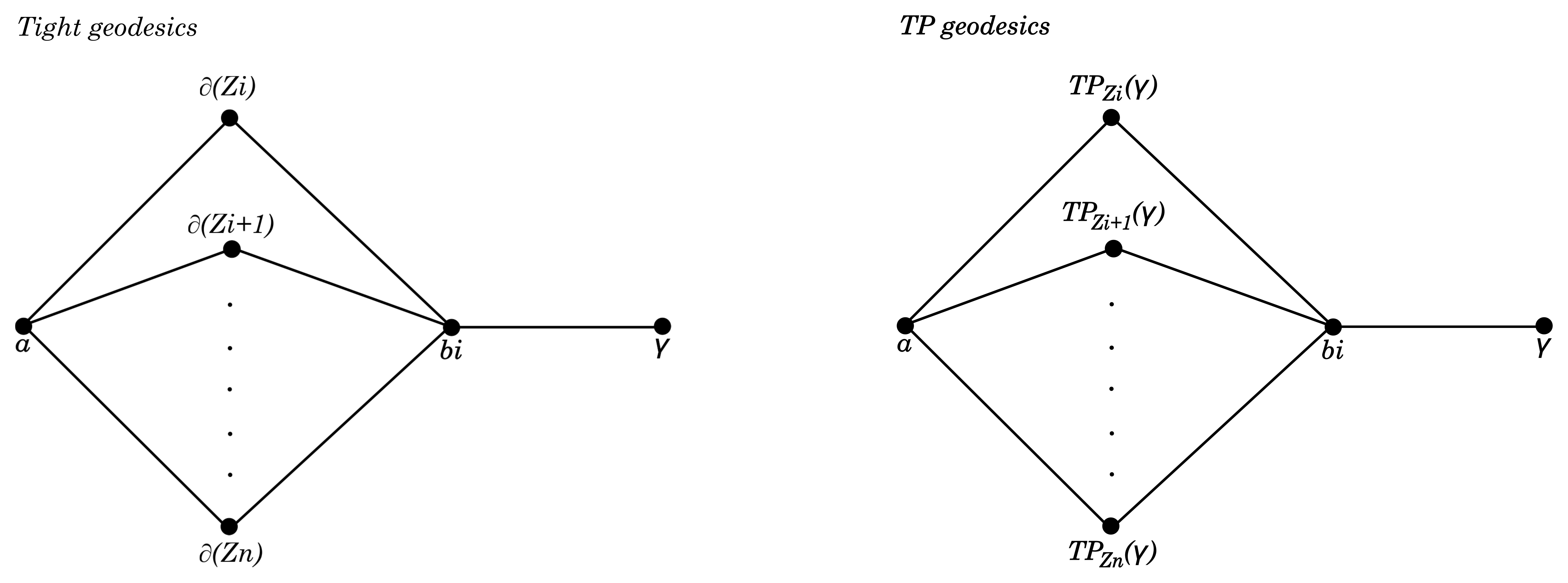}
 \caption{Tight geodesics (left) and \Projection geodesics (right) between $a$ and $\gamma$ which pass through $b_{i}$.}
 \label{tightvsBP}
\end{figure}
\qed
\end{example}

We have 
\begin{corollary}
There exist $x,y\in C(S)$ such that $d_{S}(x,y)>2$ and that $\{ \L_{\PP}(x,y) \setminus \L_{T}(x,y) \}\neq \emptyset.$ In particular, $\L_{T}(x,y)\subsetneq \L_{\PP}(x,y)\subseteq \L_{WT}^{M}(x,y).$
\end{corollary}

Example \ref{lox} and Example \ref{lox3} discuss the difference between tight geodesics and \Projection geodesics for pairs of curves of short distance. In this light, it is worthwhile noting the followings:  

\begin{lemma}\label{hc1}
There exist $x_{p},y_{p}\in C(S)$ such that $d_{S}(x_{p},y_{p})=2^{p}$, $|\L_{T}(x_{p},y_{p})|=1$, and $2^{2^{p-1}}\leq |\L_{\PP}(x_{p},y_{p})|\leq (12\cdot |\chi(S)|)^{2^{p-1}}.$
\end{lemma}
\begin{proof}
We assume $S=S_{2,0}$. The arguments given here easily generalize to general surfaces. 
Let $a,b\in C(S)$ from Example \ref{lox}. We let $x_{p}=a$ and define $y_{p}$ as follows. First $y_{1}:=b$ and we inductively define $y_{p}$; we take a pseudo--Anosov map $\phi_{p-1}$ supported on $y_{p-1}^{c}$ so that $d_{y_{p-1}^{c}}(a,\phi_{p-1}(a))\gg M$, and we define $y_{p}:=\phi_{p-1}(a).$ 

Every geodesic between $a$ and $y_{p}$ needs to contain $y_{i}$ for all $0\leq i\leq p$ by Theorem \ref{BGIT}. Hence $d_{S}(a,y_{p})=2^{p}$. Furthermore, every geodesic between $a$ and $y_{p}$ passes through the unique vertex at every time $n$--point when $n$ is even. This implies that $|\L_{T}(a,y_{p})|=1$. At every time $n$--point when $n$ is odd, we get curves which arise from projections via $\Proj$. These curves contribute to $|\L_{\PP}(a,y_{p})|$ and they arise from the arcs obtained by $\{x\cap W\}$ and $\{y\cap W\}$ where $W$ is the complementary component of the subsurface filled by the adjacent curves. Note that $W$ is homeomorphic to $F(a,b)^{c}$. A topological argument shows that $\{x\cap W\}$ and $\{y\cap W\}$ are both bounded by $3\cdot |\chi(S)|$, hence when they are projected there at most $6\cdot |\chi(S)|$ curves for each set of arcs. 
\end{proof}

We note that one can easily obtain a finer lower bound in Lemma \ref{hc1} which depends only on $p$ and the topology of $S$.

The following lemma gives a sequence of pairs of curves such that the numbers of tight geodesics and \Projection geodesics between them limit to infinity, but with some definite size difference between them at each stage, which also limits to infinity. 

\begin{lemma}\label{hc2}
There exist $x_{p},y_{p}\in C(S)$ such that $d_{S}(x_{p},y_{p})=3\cdot 2^{p-1}$, $|\L_{T}(x_{p},y_{p})|=\big( \frac{\chi(S)^{2}}{8}-\frac{\chi(S)}{4}\big)^{2^{p-1}}$, and $|\L_{\PP}(x_{p},y_{p})|\geq 2^{p-1}\cdot |\L_{T}(x_{p},y_{p})| .$
\end{lemma}
\begin{proof}
We assume $S$ is closed. We adapt the notations and the setting from Example \ref{lox3}. First we note that $|\L_{T}(a,\gamma)|= \frac{\big(\frac{-\chi(S)}{2} \big)\cdot \big( \frac{-\chi(S)}{2}+1 \big)}{2}=  \frac{\chi(S)^{2}}{8}-\frac{\chi(S)}{4}$ since $n=\frac{-\chi(S)}{2}$ in Example \ref{lox3}. 

We let $x_{p}=a$ and define $y_{p}$ as follows. First $y_{1}:=\gamma$ and we inductively define $y_{p}$; we take a pseudo--Anosov map $\phi_{p-1}$ supported on $y_{p-1}^{c}$ so that $d_{y_{p-1}^{c}}(a,\phi_{p-1}(a))\gg M$, and we define $y_{p}:=\phi_{p-1}(a).$ 
As in the proof of Lemma \ref{hc1}, we have $d_{S}(a,y_{p})=3\cdot 2^{p-1}$, $|\L_{T}(x_{p},y_{p})|=\big( \frac{\chi(S)^{2}}{8}-\frac{\chi(S)}{4} \big)^{2^{p-1}}$, and $|\L_{\PP}(x_{p},y_{p})|\geq 2^{p-1}\cdot |\L_{T}(x_{p},y_{p})| .$

The arguments given here directly apply to general surfaces by modifying Example \ref{lox3} for general surfaces by constructing an appropriate sequence of subsurfaces therein.
\end{proof}



\titleformat{\section}{\large\bfseries}{\appendixname~\thesection .}{0.5em}{}
\begin{appendices}

\section{Expansion on \projection geodesics}\label{PPRR2}
The aim of this section is to expand \projection geodesics by providing methods to produce more curves to be used for replacement of the vertices of a given thick geodesic. To do this, we use some facts about the relationship between intersection numbers of two curves and their subsurface projection distances. The methods add to the curves obtained in $\S\ref{PPRR}$, so the new classes of weak tight geodesics strictly contain \projection geodesics, but we can still control the classes of weak tight geodesics that they are contained in.
Lastly, we note that this section is dedicated to only provide the methods to produce curves to be used for replacement. One can obtain the sets of geodesics using the curves given here as in Definition \ref{pg}.

We start with the following basic fact.
\begin{lemma}\label{four}
Let $x,y\in C(S)$. Let $\mu(n):=2 \cdot n+1.$
\begin{itemize}
\item $d_{S}(x,y)\leq   \mu(i(x,y)).$
\item $d_{Z}(x,y)\leq   \mu(4 \cdot i(x,y)+4)$ for all $Z\subseteq S.$
\end{itemize}
\end{lemma}
\begin{proof}
The first statement is well--known. The second statement easily follows if $Z$ is an annulus. If $Z$ is not an annulus, it suffices to observe that $$\max\{ i(a,b)|a\in \pi_{Z}(x), b\in \pi_{Z}(y)\} \leq 4\cdot i(x,y)+4.$$ The above inequality holds since $a$ and $b$ intersect at most $4$ times around every intersection of $x$ and $y$, so this can be measured by $4\cdot i(x,y)$, the other $4$ comes from intersections which could possibly occur in the regular neighborhoods of $\partial(Z)$.
\end{proof}

Now, we define the following.
\begin{definition}\label{expdef}
Let $A\subseteq S$. Define $\Exp_{A}^{n}:C(S)\longrightarrow C(S)$, where $$ \Exp_{A}^{n}(x):=  \bigcup_{Y\in A^{nac}} \bigg\{c\in C(Y) \bigg|i(c,b)\leq n \mbox{ for all } b\in \{ \Proj_{A}(x)\cap C(Y) \} \bigg\}.$$
We note that $\{ \Exp_{A}^{n}(x)\setminus \Proj_{A}(x) \}\neq \emptyset$ for a sufficiently large $n$ as far as $x$ essentially intersects with a non--pants component of $A^{nac}$. See Corollary \ref{non--empty}. Let $C\subseteq C(S)$. We define $\Exp_{A}^{n}(C):=\cup_{x\in C}\Exp_{A}^{n}(x).$ 
\end{definition}

We observe 
\begin{lemma}\label{dehntwists}
Let $C\subseteq C(S)$ such that $i(a,b)\leq q$ for all $a,b\in C$. For all $p \in \mathbb{N}$, there exists $\gamma \in C(S)$ such that 
\begin{itemize}
\item $i(\gamma,c)\leq \max\{ q+ p\cdot q^{2}, p \cdot 4\} $ for all $c \in C$.
\item $i(\gamma,c)\geq p$ for some $c \in C.$
\end{itemize}
\end{lemma}
\begin{proof}
We refer the following inequalities on intersection numbers involving Dehn twists to \cite{FM}.
We consider two cases. 

Case 1: If there exist $a,b \in C$ such that $i(a,b)> 0$ then we have 
\begin{itemize}
\item$p \leq i(b,T_{a}^{p}(b)) =p \cdot i(a,b)^{2}\leq p\cdot  q^{2}.$ 
\item $i(c,T_{a}^{p}(b))\leq i(c,b)+p \cdot i(a,b)\cdot i(c,a)\leq q+ p\cdot  q^{2}$ for all $c \in C\setminus \{b\}.$ 
\end{itemize}
Lastly, we let $\c= T_{a}^{p}(b).$

Case 2: If $C$ is a simplex then we take $a\in C$ so that there exists $S'\subseteq S$ such that $a\in C(S')$, $\x(S')=1$, and $c\notin C(S')$ for all $c\in C \setminus \{a\}$. Then, we can find $x\in C(S')$ such that $i(x,a)\leq 2$ since $S'=S_{1,1}$ or $S_{0,4}.$ We have 
\begin{itemize}
\item $p \leq i(a,T_{x}^{p}(a)) =p \cdot i(a,x)^{2} \leq p \cdot 4.$
\item $i(c,T_{x}^{p}(a))=0$ for all $c \in C\setminus \{a\}$ since $T_{x}^{p}(a)\in C(S').$
\end{itemize}
Lastly, we let $\gamma= T_{x}^{p}(a).$
\end{proof}

We observe 
\begin{corollary}\label{non--empty}
In Definition \ref{expdef}, if $x$ essentially intersects with a non--pants component of $A^{nac}$ then $\{ \Exp_{A}^{n}(x) \setminus \Proj_{A}(x) \} \neq \emptyset$ for all $n\geq 4+5\cdot 4^{2}.$  
\end{corollary}
\begin{proof}
Take a non--pants component $Y\in A^{nac}$ which $x$ essentially intersects with. By the proof of Lemma \ref{four}, it follows that $i(a,b)\leq 4$ for all $a,b \in \pi_{Y}(x)$. Now, we take $\gamma \in C(Y)$ given by Lemma \ref{dehntwists} when $p=5$, so that we can guarantee $\gamma \in \{ \Exp_{A}^{n}(x) \setminus \Proj_{A}(x) \}.$
\end{proof}

Now, we prove the following lemma which is analogous to Lemma \ref{weekend}.

\begin{lemma}\label{weakenextra2}
Let $x,y\in C(S)$ such that $d_{S}(x,y)>2$. Let $\{x=v_{0}, v_{1}, \cdots ,v_{d_{S}(x,y)}=y\}$ be a thick geodesic connecting $x$ and $y$. Let $n\in \mathbb{N}.$ The following holds for all $i$:  
\begin{itemize}
\item Replacing $v_{i}$ by any subset of $\Exp_{F(v_{i-1}\cup v_{i+1})}^{n}( x \cup y ) $ gives a new (possibly same) thick geodesic.
\item If $\c \in \Exp_{F(v_{i-1}\cup v_{i+1})}^{n}( x \cup y ) $ and $\pi_{Z}(\gamma)\neq \emptyset$ where $Z\subsetneq S$, then $d_{Z}(x,\gamma)\leq \max\{ M, \mu(4n+5)\} \text{ or }d_{Z}(\gamma,y)\leq \max\{ M,\mu(4n+5)\}.$
\end{itemize}
\end{lemma} 
\begin{proof}
The first statement is obvious. For the second statement, we assume $\c \in  \Exp_{F(v_{i-1}\cup v_{i+1})}^{n}(x)$.

Case 1: Suppose $\pi_{Z}(v_{i-1})= \emptyset$ and $\pi_{Z}(v_{i+1})= \emptyset$. Then $Z \subseteq W \in (F(v_{i-1}\cup v_{i+1}))^{nac}$. By using Lemma \ref{weekend} and Lemma \ref{four} we have $$d_{Z}(x,\c)\leq d_{Z}(x,\pi_{W}(x))+d_{Z}(\pi_{W}(x),\gamma)\leq 2+\mu(4n+4)\leq \mu(4n+5).$$

Case 2: Suppose $\pi_{Z}(v_{i-1})\neq \emptyset$ or $\pi_{Z}(v_{i+1})\neq \emptyset$. As before, we have $d_{Z}(x,\gamma)\leq M$ or $d_{Z}(\gamma,y)\leq M.$
\end{proof}

As before, we can easily generalize Lemma \ref{weakenextra2}. 

\begin{lemma}\label{hig}
Let $x,y\in C(S)$ such that $d_{S}(x,y)>2$. Let $\{x=v_{0}, v_{1}, \cdots ,v_{d_{S}(x,y)}=y\}$ be a thick geodesic connecting $x$ and $y$. Let $n\in \mathbb{N}.$ The following holds for all $i$:  
\begin{itemize}
\item Replacing $v_{i}$ by any subset of $\Exp_{F(v_{i-1}\cup v_{i+1})}^{n} \big( N_{i-1}(x) \cup N_{d_{S}(x,y)-i-1}(y) \big)$ gives a new (possibly same) thick geodesic.
\item If $\c \in \Exp_{F(v_{i-1}\cup v_{i+1})}^{n}\big( N_{i-1}(x) \cup N_{d_{S}(x,y)-i-1}(y) \big) $ and $\pi_{Z}(\gamma)\neq \emptyset$ where $Z\subsetneq S$, then $d_{Z}(x,\gamma)\leq M+ \mu(4n+5) \text{ or }d_{Z}(\gamma,y)\leq M+ \mu(4n+5).$
\end{itemize}

\end{lemma} 
\begin{proof}
It follows by the arguments given in the proofs of Lemma \ref{weaken} and Lemma \ref{weakenextra2}.
\end{proof}

\begin{remark}\label{SCB}
We can obtain classes of weak tight geodesics using Lemma \ref{weakenextra2} and Lemma \ref{hig}, which we respectively denote by $\L_{\PP(n)}(x,y)$ and $\L_{\EE(n)}(x,y)$, as we defined \projection geodesics in Definition \ref{pg}.
\end{remark}

We have 
\begin{corollary}\label{stbk}
Let $x,y\in C(S)$ such that $d_{S}(x,y)>2$. 
\begin{itemize}
\item $\L_{T}(x,y)\subseteq \L_{\PP}(x,y)\subseteq \L_{\PP(n)}(x,y)\subseteq \L_{WT}^{\max\{M, \mu(4n+5)\}}(x,y).$
\item $\L_{T}(x,y)\subseteq \L_{\PP}(x,y)\subseteq  \L_{\EE}(x,y) \subseteq \L_{\EE(n)}(x,y)\subseteq \L_{WT}^{M+ \mu(4n+5)}(x,y).$
\end{itemize}
\end{corollary}

By combining Example \ref{lox} and Corollary \ref{non--empty}, we have
\begin{corollary}
For all $n\geq 4+5\cdot 4^{2}$, there exist $x,y\in C(S)$ such that $d_{S}(x,y)>2$ and that $\{ \L_{\PP(n)}(x,y) \setminus \L_{\PP}(x,y) \}\neq \emptyset$ and $\{ \L_{\EE(n)}(x,y) \setminus \L_{\EE}(x,y) \}\neq \emptyset$.
\end{corollary}


\section{Narrow gaps of weak tight geodesics}\label{NARROW}
The aim of this section is to consider the gap between a general pair of classes of weak tight geodesics. For this, we only deal with special pairs of curves such that Corollary \ref{badpair2} holds when $i=1$ or $i=d_{S}(x,y)-1$. On these special pairs of curves, replacement procedure works effectively to guarantee the existence of geodesics which live in narrow gaps of weak tight geodesics between them. 





First, we observe the following.
\begin{lemma}\label{sincl}
Let $W\subsetneq S$ such that $\x(W)>0$. Let $y\in C(S)$ such that the arcs obtained by $\{y\cap W\}$ fill $W$. There exists $\gamma \in C(W)$ such that $$k-1\leq \max_{Z\subseteq W}d_{Z}(\gamma, y) \leq k+1$$ for all $k\in \mathbb{N}_{\geq \4}$.
\end{lemma}
\begin{proof}
We prove by the induction on the complexity of $W$. Let $\{a_{i}\}$ denote the arcs obtained by $\{y\cap W\}$. 

Suppose $\x(W)=1$. We assume $W=S_{1,1}$. (The same argument applies if $W=S_{0,4}$.) We may assume $\pi_{W}(\{a_{i}\})=\{0, \infty\}$ or $\{0, 1,\infty\}$ where the vertices of $C(W)$ are identified with $\mathbb{Q}.$ For simplicity, we assume the former. (The same argument aplies to the latter.) We take $\gamma=T_{\infty}^{n}(0)=n$. Then we have 
\begin{itemize}
\item $d_{Z}(\gamma, \{0, \infty\})\leq \4$ for all $Z\subseteq W$ except for $Z=R(\infty)$ since $d_{W}(\gamma, \infty)=d_{W}(0, \infty)=1$, see \cite{MM2}.
\item $d_{R(\infty)}(\gamma, \{0, \infty\})=d_{R(\infty)}(\gamma, \{0\})=|n|+2$, see \cite{MM2}. 
\end{itemize}
Therefore, $\text{for all $k\in \mathbb{N}_{\geq\4}$} $, $\displaystyle \max_{Z\subseteq W}d_{Z}(\gamma, \pi_{W}(\{a_{i}\}))= k.$  By arguments given in the proof of Lemma \ref{weekend}, we observe that $$k-1\leq \max_{Z\subseteq W}d_{Z}(\gamma, y) \leq k+1.$$

Suppose $\x(W)>1$. Let $Y$ be a complementary component of $a_{1}$ in $W$. Then $\x(Y)<\x(W)$ and $\{y\cap Y\}$ still fills $Y$. By the inductive hypothesis, there exists $\gamma \in C(Y)$ such that $\displaystyle k-1\leq \max_{Z\subseteq Y}d_{Z}(\gamma, y) \leq k+1.$ Now, if $Z\subseteq W$ and $Z\nsubseteq Y$ then $\pi_{Z}(\pi_{W}(a_{1}))\neq \emptyset$. Since $i( \gamma, \pi_{W}(a_{1}))=0$, we have $d_{Z}(\gamma,\pi_{W}(a_{1}) )\leq 2$. By the proof of Lemma \ref{weekend}, we have $d_{Z}(\pi_{W}(a_{1}),y)\leq 2$. Hence, $d_{Z}(\gamma,y ) \leq d_{Z}(\gamma,\pi_{W}(a_{1}) ) +d_{Z}(\pi_{W}(a_{1}),y ) \leq 4$. Therefore, we have $$k-1\leq \max_{Z\subseteq W}d_{Z}(\gamma, y) \leq k+1$$ for all $k\in \mathbb{N}_{\geq\4}$.

\end{proof}

By using Lemma \ref{sincl}, we have   
\begin{corollary}\label{sincl2}
Let $x,y\in C(S)$ such that $d_{S}(x,y)>2$ and that Corollary \ref{badpair2} holds when $i=1$ or $i=d_{S}(x,y)-1$. Then $\{ \L_{WT}^{m}(x,y) \setminus \L_{WT}^{n}(x,y) \} \neq \emptyset \mbox{ for all }n,m\in \mathbb{N}_{\geq M}$ such that $m-n>2.$
\end{corollary}
\begin{proof}
We assume Corollary \ref{badpair2} holds when $i=1$ where $x=v_{0}$ and $y=v_{d_{S}(x,y)}$. We let $W$ denote a complementary component of $F(v_{0}\cup v_{2})$ whose complexity is bigger than $0$. Note that $\pi_{W}(x)=\emptyset.$
Since $d_{S}(x,y)>2$, the arcs obtained by $\{y \cap W\}$ fill $W$. Therefore, we can apply Lemma \ref{sincl}: there exists $\gamma \in C(W)$ such that $$n< \max_{Z\subseteq W}d_{Z}(\gamma,y)\leq n+3$$ for all $n\in \mathbb{N}_{\geq M}$.
Now, we construct a geodesic in $\{ \L_{WT}^{m}(x,y) \setminus \L_{WT}^{n}(x,y) \}$ where $m-n>2$; 
we replace $v_{1}$ by $\gamma$, and connect $\gamma$ and $y$ by a tight geodesic, a new geodesic is a desired geodesic since it is 
\begin{itemize}
\item $(n+3)$--weakly tight at time $1$--point but not $n$--weakly tight. (If $Z\nsubseteq W$ then $d_{Z}(x,\gamma)\leq M$ or $d_{Z}(\gamma, y)\leq M$ by arguments given in Lemma \ref{weekend}.) 
\item $M$--weakly tight at time $i$--point for all $i>1$ by tightness.  
\end{itemize}
\end{proof}

\end{appendices}

\bibliographystyle{abbrv}
\bibliography{references.bib}

\end{document}